\documentclass[11pt,reqno]{amsart}
\usepackage{amsmath,amsfonts,amsthm,amssymb,geometry,graphics,enumerate,newlfont,latexsym,tabularx,shapepar}

\usepackage{enumerate}
\usepackage{pinlabel}
\newtheorem{thm}{Theorem}[section]
\newtheorem{defin}[thm]{Definition}
\newtheorem{lemme}[thm]{Lemma}

\newtheorem{prop}[thm]{Proposition}
\newtheorem{cor}[thm]{Corollary}
\newtheorem{remark}{Remark}
\newtheorem{step}{Step}

\DeclareMathOperator{\lip}{Lip}
\DeclareMathOperator{\capa}{cap}
\DeclareMathOperator{\dime}{dim}
\DeclareMathOperator{\loc}{loc}
\DeclareMathOperator{\supp}{supp}
\DeclareMathOperator{\ric}{Ric}

\DeclareMathOperator{\dom}{Dom}
\DeclareMathOperator{\lipp}{lip}

\DeclareMathOperator{\CD}{CD}
\DeclareMathOperator{\RCD}{RCD}
\DeclareMathOperator{\MCP}{MCP}
\begin{document}
\title[Liouville-type Theorems]{Riemannian Polyhedra and Liouville-type Theorems for Harmonic maps}

\author{Zahra Sinaei}
\begin{abstract}
 This paper is a study of harmonic maps from Riemannian polyhedra to locally non-positively curved geodesic spaces in the sense of Alexandrov.  We prove Liouville-type theorems for subharmonic functions and harmonic maps under two different assumptions on the source space. First we prove the analogue of the Schoen-Yau Theorem on a complete  pseudomanifolds with non-negative Ricci curvature. Then we study $2$-parabolic admissible Riemannian polyhedra and prove some vanishing results on them.
\end{abstract}
\maketitle

\section{Introduction}
 Harmonic maps between singular spaces have received considerable attention since the early 1990s. Existence of energy minimizing locally Lipschitz maps from Riemannian manifolds into Bruhat-Tits buildings and Corlette's version of Margulis's
super-rigidity theorem were proved in \cite{GS92}. In \cite{KS93} Korevaar and Schoen constructed harmonic
maps from domains in Riemannian manifolds into Hadamard spaces as a boundary
value problem. The book \cite{F01} by Eells and Fuglede contains a description of the application of the methods of \cite{KS93,Jost94,Jost95,Jost97,KorevaarSchoen97,Jost98} to the study of maps between polyhedra, see also \cite{C95,DM08,DM10}.

Our first objective in this paper is to prove Liouville-type theorems for harmonic maps. We prove the analogue of the Schoen-Yau Theorem on complete (smooth) pseudomanifolds with non-negative Ricci curvature. To this end, we generalize some Liouville-type theorems for subharmonic functions from \cite{Y82}.

The classical Liouville theorem for functions on manifolds states that on a complete Riemannian manifold with non-negative Ricci curvature, any harmonic function bounded from one side must be a constant. In \cite{Y82}, Yau proves  that there is no non-constant, smooth, non-negative,  $L^p$, $p>1$, subharmonic function on a complete Riemannian manifold.  He also proves that  every continuous subharmonic function defined on a complete Riemannian manifold whose local Lipschitz constant is bounded by $L^1$-function is also harmonic. Furthermore if the $L^1$-function belongs to $L^2$ as well, and the manifold has non-negative Ricci curvature, then the subharmonic function is constant. In the smooth setting, there are two types of assumptions that have been studied on the Liouville property of harmonic maps. One is the finiteness of the energy and the other is the smallness of the image. For example, Schoen and Yau in \cite{SY76}, proved that any non-constant harmonic map from a complete non-compact manifold of non-negative Ricci curvature to a manifold of non-positive sectional curvature has infinite energy. Hildebrandt-Jost-Widman \cite{H80} (see also \cite{H82,H85}) proved a Liouville-type theorem for harmonic maps into regular geodesic balls in a complete $C^3$-Riemannian manifold from a simple or compact $C^1$-Riemannian
manifold.  For more references for Liouville-type theorems for harmonic maps and functions in both smooth and singular setting see the introduction in \cite{S08}.

 A connected locally finite $n$-dimensional simplicial polyhedron $X$ is
called {\it{admissible}},  if $X$ is dimensionally $n$-homogeneous and  $X$ is locally $(n -1)$-chainable. It is called {\it{circuit}} if instead it is $(n-1)$-chainable and every $(n-1)$-simplex is the face of at most two $n$-simplex and {\it{pseudomanifold}} if it is admissible circuit. A polyhedron $X$ becomes a Riemannian polyhedron when endowed with a Riemannian structure $g$, defined by giving on each maximal simplex $s$ of $X$ a bounded measurable Riemannian metric $g$  equivalent to a Euclidean metric on $s$.

There exist slightly various notions of boundedness of  Ricci curvature from below on general metric spaces. See for example  \cite{S06, V09,O07,KS01,KS03,AGS12,EKS13,AMS13} and the references therein. In what follows by $\ric_{N,\mu_g}\geq K$ we mean that $(X,g,\mu_g)$ satisfies the either $\MCP(K,N)$ or $\CD(K,N)$ unless otherwise specified. 
As these definitions are somewhat technical we refer the reader to Section \ref{Ricci} for a precise statement.

The definition of harmonic maps from admissible Riemannian polyhedra to metric spaces is similar to the one in the smooth setting. However due to lack of smoothness some care is needed in defining the notions of energy density, the energy functional and energy minimizing maps. Precise definitions and related results can be found in Subsection \ref{HMORM}.

We can state now the main results which we obtain in this direction.
\begin{thm}\label{yau1}
Suppose $(X,g,\mu_g)$ is a complete, admissible Riemannian polyhedron and  $f\in W^{1,2}_{\loc}(X)\cap L^2(X)$ is a non-negative, weakly subharmonic function. Then $f$ is constant.
\end{thm}

\begin{thm}\label{pseudo}
Let $(X,g,\mu_g)$ be a complete,  smooth $n$-pseudomanifold. Suppose $X$ satisfies $\CD(0,n)$. Let $f$  be a continuous, weakly subharmonic function  belonging to $ W^{1,2}_{\loc}(X)$ such that both $\|\nabla f\|_{L^1}$ and $\|\nabla f\|_{L^2}$ are finite. Then $f$ is a constant function.
\end{thm}
Here by a smooth pseudomanifold we mean a simplexwise smooth, pseudomanifold which is  smooth  outside of its singular set. That situation arises for instance when the space is a projective algebraic variety.  The difficulty in extending existing results lies in the lack of a differentiable structure on admissible polyhedron in general, and the loss of completeness outside the singular set even in the case of smooth pseudomanifolds. Moreover the classical notion of Laplace operator doesn't exist in the non-smooth setting. To circumvent this latter problem and following the work of \cite{G12}, we define the Laplacian of a subharmonic function  as a  measure for which the Green formula holds, see Theorem \ref{radon}. Furthermore we prove a Gaffney's Stokes theorem in this setting, see Theorem \ref{sub-har}.

 The following two theorems are corollaries of Theorems \ref{yau1} and \ref{pseudo} and are the generalizations of Shoen-Yau's theorem in the non-smooth setting.
\begin{thm}\label{rigid4}
Let $(X,g,\mu_g)$ be a complete, smooth $n$-pseudomanifold. Suppose $X$ has non-negative
$n$-Ricci curvature. Suppose $Y$ is a  Riemannian manifold of non-positive curvature, and $u:(X,g)\rightarrow (Y,h)$ a continuous harmonic map belonging to $W_{loc}^{1,2}(X, Y)$. If $u$ has finite energy and $e(u)$ is locally bounded, then it is a constant map.
\end{thm}
\begin{thm}\label{rigid5}
Let $(X,g,\mu_g)$ be a complete, smooth $n$-pseudomanifold. Suppose $X$ has non-negative $n$-Ricci curvature $\CD(0,n)$. Let $Y$ be a simply connected, complete geodesic space of non-positive curvature and $u:(X,g)\rightarrow Y$ a continuous harmonic map with finite energy, belonging to $W^{1,2}_{\loc}(X, Y)$. If  $\int_M\sqrt{e(u)}~d\mu_g<\infty$, then $u$ is a constant map.
\end{thm}

Our second objective in this paper is the study of  $2$-parabolic admissible polyhedra. We say a connected domain $\Omega$ in an  admissible Riemannian polyhedron is $2$-parabolic, if for every compact set in $\Omega$, its relative capacity with respect to $\Omega$ is zero. Our main theorem is
\begin{thm}\label{para-har}
Let $X$ be $2$-parabolic pseudomanifold. Let $f$ in $W^{1,2}_{\loc}(X)$ be a continuous, weakly subharmonic function  such that $\|\nabla f\|_{L^1}$ and $\|\nabla f\|_{L^2}$ are finite. Then $f$ is constant.
\end{thm}
Just as in the case of complete pseudomanifolds
\begin{thm}\label{rigid6}
Let $(X,g,\mu_g)$ be a $2$-parabolic pseudomanifold with $g$ simplexwise smooth. Let $Y$ be a simply connected complete geodesic space of non-positive curvature and $u:(X,g)\rightarrow Y$ a continuous harmonic map with finite energy belonging to $ W^{1,2}_{\loc}(X, Y)$. If we have $\int_X\sqrt{e(u)}d\mu_g<\infty$, then $u$ is a constant map.
\end{thm}
 Also we will obtain
\begin{thm}\label{rigid7}
Let $(X,g,\mu_g)$ be a $2$-parabolic admissible Riemannian polyhedron with simplexwise smooth metric $g$. Let $Y$ be a complete geodesic space of non-positive curvature and $u:(X,g)\rightarrow Y$ a continuous harmonic map  belonging to $ W^{1,2}_{\loc}(X, Y)$, with bounded image. Then $u$ is a constant map.
\end{thm}
In order to prove  Theorem  \ref{para-har}, we need to generalize some of the results in \cite{H90} in our setting. This is done in Section \ref{parabolic} in Propositions \ref{holop1}, \ref{holop2} and \ref{para-har2}. The proof of  Propositions  \ref{holop1} and  \ref{holop2} follow a similar pattern  as their equivalents for Riemannian manifolds. They are based on the fact that admissible Riemannian polyhedra admit exhaustions by regular domains and the validity of comparison principle on admissible Riemannian  polyhedra. The main new ingredient in the proof of Theorem \ref{para-har} is Proposition \ref{para-har2}.

  The rest of this paper is organized as follows. In  Section \ref{prelim}, we give a complete background on Riemannian polyhedra and analysis on them. Most definitions and results have been taken directly from \cite{F01}. In Subsection \ref{sobolev}, we compare the $L^2$ based Sobolev space on admissible Riemannian polyhedra  as in \cite{F01} with the one in \cite{C99}, and  show that they are equivalent. As we could not find references in the literature we provide a rather detailed explanation of this fact. In Section \ref{Ricci}, we discuss the definition of two notions of Ricci curvature, the  curvature dimension condition $\CD(K,N)$ and the measure contraction property $\MCP(K,N)$ on metric measure spaces. We show that both notions are applicable to Riemannian polyhedra. In Proposition \ref{bishop} we  show that any non-compact complete $n$-dimensional Riemannian polyhedron of non-negative Ricci curvature has infinite volume. Subsection \ref{complete1} is devoted to Theorems \ref{yau1}, \ref{sub-har}, \ref{pseudo} and Subsection \ref{vanishing} to Theorems \ref{rigid4}, and \ref{rigid5}. In Section \ref{parabolic} we show that as in the smooth case the ``approximation by unity'' property holds on admissible $2$-parabolic polyhedra, see Lemma \ref{approx}.  Moreover, we prove that removing the singular set of a $2$-parabolic pseudomanifold yields  a $2$-parabolic manifold (Lemma \ref{sing-para}). The rest of this Section is the detailed proof of Theorem \ref{para-har} and its corollaries.
\section*{Acknowledgement}
The original version of this work was part of my Ph.D. dissertation. I thank my advisor Professor Marc Troyanov for his guidance and support. I also thank my Ph.D. exam jury members Professors Buser, Wenger and particularly Professor Naber who suggested an additional bound on $|\nabla f|$ as hypothesis in Theorem \ref{pseudo},
and so automatically on $e(u)$ in  Theorems \ref{rigid4} and \ref{rigid5}. I am very grateful to thank Professor Goufang Wei who I met as a postdoc in MSRI, for suggesting that I read papers by \cite{HKX13,J13,K13} which allowed me to remove this assumption in Theorems \ref{pseudo}, \ref{rigid5}.\\
\section{Preliminaries}\label{prelim}
\subsection{Riemannian polyhedra}\label{Riempoly}
In this subsection we recall the definitions and results about Riemannian polyhedra which will be used in the rest of the manuscript. We refer the reader to the book \cite{F01} and the references therein, for more complete discussion on the subject.
\subsubsection*{\textbf{Simplicial complex.} }A countable locally finite simplicial complex $K$,
consists of a countable set $\{v\}$ of elements called vertices and a set $\{s\}$ of
finite non-void subsets of vertices called simplexes such that
\begin{enumerate}[i.]
  \item any set consisting of exactly one vertex is a simplex.
  \item any non-void subset of a simplex is a simplex.
  \item every vertex belongs to only finitely many simplexes (the local finiteness
condition).
\end{enumerate}

To the simplicial complex $K$, we associate a metric space $|K|$ defined as follows. The space $|K|$ of $K$ is the set of all formal finite linear combinations
$\alpha=\sum_{v\in K} \alpha(v)v$ of vertices of $K$ such that $0\leq\alpha(v)\leq 1$, $\sum_{v\in K} \alpha(v)=1$
and $\{v : \alpha(v) > 0\}$ is a simplex of $K$. $|K|$ is made into a metric space with barycentric distance $\rho(\alpha, \beta)$ between
two points $\alpha=\sum \alpha(v)v$ and $\beta=\sum \beta(v)v$ of $|K|$ given by the finite sum
\begin{eqnarray*}
\rho(\alpha, \beta) = \left(\sum _{v\in K}(\alpha(v) -\beta(v))^2\right)^{\frac{1}{2}}.
\end{eqnarray*}
With this metric $|K|$ is locally compact and  separable. The metric $\rho$ is not intrinsic. We denote by $\overline{\rho}(\alpha,\beta)$ the length metric associated to $\rho$ by the standard procedure \cite{BBI01}.
\begin{lemme}\label{embedding}\cite{F01}
Let $K$ be a countable, locally finite simplicial complex of finite
dimension $n$ and $V$ a Euclidean space of dimension $2n + 1$. There exists an
affine Lipschitz homeomorphism $i$ of $|K|$ onto a closed subset of $V$.
\end{lemme}

We shall use the term {\it{polyhedron}} to mean a connected locally compact separable
Hausdorff space $X$ for which there exists a simplicial complex $K$ and
a homeomorphism $\theta$ of $|K|$ onto $X$. Any such pair $T = (K,\theta)$ is called a
triangulation of $X$.

 A {\it{Lipschitz polyhedron}} is a metric space $X$ which is the image of the metric space $|K|$
of some complex $K$ under a Lipschitz homeomorphism $\theta : |K| \rightarrow  X$. The pair
$(K, \theta)$ is then called a Lipscitz triangulation (or briefly a triangulation) of the
Lipschitz polyhedron $X$. A null set in a Lipschitz polyhedron $X$ is understood a set $Z \subset X$ such
that $Z$ meets every maximal simplex $s$ (relative to some, and hence any
triangulation $T = (K, \theta)$ of $X$) in a set whose preimage under $\theta$ has $p$-dimensional
Lebesgue measure $0$, $p = \dime s$.

 From Lemma \ref{embedding} follows that every Lipschitz polyhedron $(X,d_X)$
can be mapped Lipschitz homeomorphically and (simplexwise) affinely onto a
 closed subset of a Euclidean space.
\subsubsection*{\textbf{Riemannian Structure on a polyhedron.}} The class of domains that we consider for our harmonic maps are Riemannian polyhedra. A Riemannian polyhedron is a Lipschitz polyhedron  $(X,d)$ such that for some triangulation $T = (K, \theta)$, there exist a measurable Riemannian metric $g^s=g_{ij}dx^idx^j$ on each maximal simplex $s$ of $i(|K|)$ ($i$ as in  Lemma \ref{embedding}), which satisfies
\begin{eqnarray}\label{elliptic}
\Lambda^{-2} \|\xi\|^2 \leq g_{ij}(x)\xi^i\xi^j \leq \Lambda^{2} \|\xi\|^2
\end{eqnarray}
almost everywhere in standard coordinate in the simplex $s$. Here the constant $\Lambda$ is independent of a given simplex.
The distance $d^g_X$ on $X$ is an intrinsic distance with respect to the metric $g$, meaning that $d^g=d_X^g$ is the infimal length of admissible path joining $x$ to $y$. Actually $(X,d^g)$ is a length space. The detailed definition is somewhat subtle and we refer to \cite{F01}, for a careful discussion of Riemannian polyhedra.

A Riemannian metric $g$ on a polyhedron $X$ is said to be continuous, if
relative to some (hence any) triangulation, $g_s$ is continuous up to the boundary
on each maximal simplex $s$ and for any two maximal simplexes $s$ and $s'$
sharing a face $t$, $g_s$ and $g_{s'}$ induce the same Riemannian metric on $t$. There
is a similar notion of a Lipschitz continuous Riemannian metric.

A Riemannian polyhedron has a well defined volume element given simplexwise by
\begin{eqnarray*}
d\mu_g=\sqrt{\det(g_{ij}(x))}~dx_1dx_2\ldots dx_n,
\end{eqnarray*}
this measure coincide with Hausdorff measure.
\subsubsection*{\textbf{Further definitions.}}
 A polyhedron $X$ will be called {\it{admissible}} if in some (hence in any) triangulation
\begin{enumerate}[i.]
\item $X$ is dimensionally homogeneous, i.e. all maximal simplexes have the
same dimension $n(= \dim X)$, or equivalently every simplex is a face of
some $n$-simplex.
\item $X$ is locally $(n - 1)$-chainable, i.e.  for every connected open set $U \subset X$, the open set $U \backslash X^{n-2}$ is connected.\\
 \end{enumerate}
 The boundary $\partial X$ of a polyhedron $X$ is the union of all non-maximal simplexes contained in only one maximal simplex. In this work we always assume that $(X,g)$ satisfies  $\partial X=\emptyset$.

 By an {\it{$n$-circuit}} we mean a polyhedron $X$ of homogeneous dimension $n$ such
that in some, (and hence any) triangulation,
\begin{enumerate}[i.]
 \item every $(n - 1)$-simplex is a face of at most two $n$-simplexes (exactly
two if $\partial X =\emptyset$).
\item $X$ is $(n - 1)$-chainable, i.e. $X \backslash X^{n-2}$ is connected, or equivalently
any two $n$-simplexes can be joined by a chain of contiguous $(n - 1)$- and
$n$-simplexes.
\end{enumerate}

Let $S = S(X)$ denote the {\it{singular}} set of an $n$-circuit $X$, i.e. the complement
of the set of all points of $X$ having a neighborhood which is a
topological $n$-manifold (possibly with boundary). $S$ is a closed triangulable
subspace of $X$ of codimension $ \geq 2$, and $X \backslash S$ is a topological $n$-manifold
which is dense in $X$. An admissible circuit is called a {\it{pseudomanifold}}. We call a pseudomanifold $(X,g,d_X)$ a Lipschitz pseudomanifold, if $g$ is Lipschitz continuous. If $g$ is simplexwise smooth such that $(X\backslash S,g|_{X\backslash S})$ has the structure of a smooth Riemannian manifold, we call $(X,g,d_X)$ a {\it{smooth pseudomanifold}}. \footnote{In many texts the term pseudomanifold is used for what we called a circuit.}
\subsection{The Sobolev space $W^{1,2}(X)$}\label{sobolev}
Let $(X,g,d_X)$ denote an admissible Riemannian polyhedron of dimension $n$.
We denote by $\lip^{1,2}(X)$ the linear space of all Lipschitz continuous functions
$u : (X, d_X) \rightarrow \mathbb R$ for which the Sobolev $(1, 2)$-norm $\|u\|$ defined by
\begin{eqnarray*}
\|u\|^2_{1,2} = \int_X (u^2 + |\nabla u|^2)~d\mu_g = \sum_{s\in S^{(n)}(X)} \int_s (u^2 + |\nabla u|^2)~d\mu_g
\end{eqnarray*}
is finite, $S^{(n)}(X)$ denoting the collection of all $n$-simplexes $s$ of $X$ and $|\nabla u|$ the Riemannian norm of the Riemannian
gradient on each $s$ (The Riemannian gradient is defined a.e. in $X$ or  a.e. in each
$s\in S^n(X)$ by Rademacher's theorem for Lipschitz functions on Euclidean domains).

The Lebesgue space $L^2(X)$ is likewise defined with respect to the volume measure.
The Sobolev space $W^{ 1,2} (X)$ is defined as the completion of $\lip^{1,2} (X)$ with
respect to the  Sobolev norm $\|\cdot\|_{1,2}$. We use the notations $\lip_c(X)$, $W^{1,2}_0(X)$, and $W^{1,2}_{\loc}(X)$ for the linear space of  functions in $\lip(X)$ with compact support, the closure of $\lip_c(X)$ in $W^{1,2}(X)$ and  all $u\in L^2_{\loc}(X)$ such that $u\in W^{1,2}(U)$ for all relatively compact subdomains $U$ in $X$.
\subsubsection*{\textbf{Sobolev spaces on  metric spaces.}} Here we recall a few basic notions on analysis on metric spaces.  For the sake of completeness, we compare the $L^2$ based Sobolev space on admissible Riemannian polyhedra  as in \cite{F01} with the one in \cite{C99} and  show that they are equivalent.  We use  \cite{C99} as our main reference. See also \cite{S00, HK98,H96,HK00, AGS11} and \cite{B11} for further references.

 Let $(Y,d,\mu)$ be a metric measure space, $\mu$ a Borel regular measure. Assume also the measure of balls of finite and positive radius are finite and positive. Fix a set $A\subset Y$. Let $f$ be a function on $A$ with values in the extended real numbers.
\begin{defin} An upper gradient  for $f$ is an extended real
valued Borel function $g : A \rightarrow [0,\infty]$ such that for all points $y_1, y_2 \in A$
and all continuous rectifiable curves $c : [0,l]\rightarrow A$ parameterized by arc length
$s$ with $c(0) = y_1$, $c(l) = y_2$, we have
\begin{eqnarray*}
|f(y_2)-f(y_1)|\leq\int_0^l g(c(s))~ds.
\end{eqnarray*}
\end{defin}
Note that in above definition the left-hand side is interpreted as $\infty$, if either $f(y_1)=\pm\infty$ or $f(y_2)=\pm\infty$. If on the other hand, the right-hand side is finite then it follows that $f(c(s))$ is a continuous function of $s$.
For  a Lipschitz function $f$ we define the lower pointwise Lipschitz constant of $f$ at $x$ as
\begin{eqnarray*}
\lipp f(x)=\liminf_{r\rightarrow 0}\sup_{y\in B(x,r)}\frac{|f(y)-f(x)|}{r}.
\end{eqnarray*}
 $\lipp f$ is  Borel, finite and bounded by the Lipschitz constant. Also $\lipp f$ is an upper gradient for $f$. Similarly for Lipschitz function $f$, the upper pointwise Lipschitz constant $f$, $\lip f$, is the Borel function
\begin{eqnarray*}
\lip f(x)=\limsup_{r\rightarrow 0}\sup_{y\in B(x,r)}\frac{|f(y)-f(x)|}{r}.
\end{eqnarray*}
For  any Lipschitz function $f$, we have $\lipp f(x)\leq \lip f$. In the special case $Y=\mathbb R^n$, if $x$ is a point of differentiability
of $f$, we observe that $\lipp f(x)=\lip f(x)=|\nabla f(x)|$. We now define the Sobolev space $H^{1,p}$ for $1\leq p<\infty$.
\begin{defin}
Whenever $f\in L^p(Y)$, let
\begin{eqnarray*}
\|f\|_{1,p}=\|f\|_{L^p}+\inf_{g_i}\liminf_{i\rightarrow\infty}\|g_i\|_{L^p}
\end{eqnarray*}
where the infimum is taken over all sequence  $\{g_i\}$, for which there exists a sequence $f_i\stackrel{\tiny{L^p}}{\longrightarrow}f$, such that $g_i$ is an upper gradient for $f_i$ for all $i$.
\end{defin}
For $p\geq 1$, the Sobolev space $H^{1,p}$, is the
subspace of $L^p$ consisting of functions $f$ for which $\|f\|_{1,p} < \infty$ equipped
with the norm $\|\cdot\|_{1,p}$. The space $H^{1,p}$ is complete.

We  define now the notions of generalized upper and minimal upper gradients. This will allow us to give a nice interpretation of the $H^{1,p}$ norm of Sobolev functions.
\begin{defin} We say
\begin{enumerate}[i.]
\item A function $g \in L^p$ is a generalized upper gradient for $f \in L^p$, if there exist sequences $f_i\stackrel{\tiny{L^p}}{\longrightarrow}f$, $g_i\stackrel{\tiny{L^p}}{\longrightarrow}g$, such that $g_i$ is an upper gradient for $f_i$ for all $i$.
\item  For fixed $p$, a minimal generalized upper gradient for $f$
is a generalized upper gradient $g_f$ such that $\|f\|_{1,p}=\|f\|_{L^p}+\|g_f\|_{L^p}$.
\end{enumerate}
\end{defin}
The following theorem ensures the existence of minimal generalized upper gradient for any Sobolev function.
\begin{thm}\cite{C99}
For all $1 < p < \infty$ and $f \in H^{1,p}$, there exists a minimal
generalized upper gradient $g_f$, which is unique up to modification on
subsets of measure zero.
\end{thm}
We will discuss two important properties of metric spaces called
the {\it{ball doubling property}} and the {\it{Poincar\'e inequality}} for functions on them. These are essential assumptions to get a richer theory on metric spaces.
\begin{defin}
Let $(Y,d,\mu)$ be a metric measure space. The measure $\mu$ is said to be locally doubling if for all $r'$, there exists $\kappa =\kappa(r')$ such that for all $y\in Y$  and $0<r<r'$,
\begin{eqnarray}\label{ball-doubl}
0 <\mu(B_r(y))\leq 2^{\kappa}\mu(B_{r\slash 2}(y)).
\end{eqnarray}
\end{defin}
\begin{defin}
Let $q\geq 1$. We say that $Y$ supports a weak Poincar\'{e} inequality of type $(q,p)$ if
for all $r'>0$, there exist constants $1\leq\lambda<\infty $ and $C=C(p,r') > 0$  such that for all  $r\leq r'$ and all upper gradients $g$ of $f$,
\begin{eqnarray}\label{Poincare}
\left(\int_{B_r(x)}|f-f_{x,r}|^q~d\mu\right)^{1\slash q}\leq Cr\left(\int_{\lambda B_r(x)}|g|^p~d\mu\right)^{1\slash p}
\end{eqnarray}
where $f_{x,r}:=\int_{B_r(x)} f~d\mu$. If $\lambda=1$, then we say that $X$ supports a strong $(q,p)$-Poincar\'{e} inequality.
\end{defin}
For every admissible Riemannian polyhedron $(X,g,\mu_g)$, $\mu_g$ is locally doubling. Moreover $X$ supports a weak $(2,2)$-Poincar\'{e} inequality and by H\"{o}lder's inequality $(1,2)$-Poincar\'e inequality (see Corollary $4.1$ and Theorem $(5.1)$ in  \cite{F01}). In the sequel, the words ``Poincar\'{e} inequality'' refer to $(2,2)$-Poincar\'{e} inequality.

By Theorem $4.24$ in \cite{C99} for any metric space which satisfies (\ref{ball-doubl}) and (\ref{Poincare}), for some $1\leq p<\infty$ and $q=1$,  the subspace of locally Lipschitz functions is dense in $H^{1,p}$. Furthermore, on a locally complete metric space with the mentioned properties, for some $1<p<\infty$ and for any $f$ locally lipschitz, we have $g_f=\lip f$, $\mu$-almost everywhere (see \cite{C99} Theorem 6.1). Therefore,
 on a Riemannian polyhedron $(X,g,\mu_g)$, for any $f\in H^{1,2}$, $g_f(y)=|\nabla f(y)|$ for a.e. $y$ and it follows that $ H^{1,2}$ is equivalent to $W^{1,2}$.

In the following, we always consider $X=(X,g,\mu_g)$ to be an admissible Riemannian polyhedron. Some of the concepts below are defined on metric spaces in general but for simplicity we  present them only on Riemannian polyhedron and for $p=2$. For more information on metric spaces we refer the reader to \cite{B11}.
\subsection{Potential theory background}
In this subsection we recall some of the definitions in potential theory.  First we define different notions of capacities (see \cite{B11}).
\subsubsection*{\textbf{Sobolev and variational capacities.}} The {\it{Sobolev capacity}} of a set $E\subset X$ is the number
\begin{align*}
C(E)=\inf\|u\|_{W^{1,2}(X)}^2
\end{align*}
where the infimum is taken over all $u\in W^{1,2}(X)$ such that $u\geq1$ on $E$.\\

Let $E\subset \Omega\subset X$ and $\Omega$ bounded. The  {\it{variational capacity}} is defined as
\begin{eqnarray}\label{capa}
\capa(E,\Omega)=\inf_u\int_{\Omega}|\nabla u|^2~d\mu_g
\end{eqnarray}
where the infimum is taken over all $u\in W^{1,2}_0(\Omega)$ such that $u\geq 1$ on $E$.  In this definitions the infimum can be taken only over $u\leq1$ such that it is equal $1$ on a neighborhood of $E$. Also we write $\capa(E)=\capa(E,X)$.\\

 A set $U\subset X$ is {\it{quasi open}} if there are open sets $\omega$ of arbitrarily small capacity such that $U\backslash \omega$ is open relative to $X\backslash \omega$. A map $\phi:U\rightarrow Y$ from a quasiopen set $U$ to a topological space $Y$ with a countable base of open sets is {\it{quasicontinuous}} if there are open sets $\omega$ of arbitrarily small capacity such that $\phi|_{U\backslash \omega}$ is continuous.
Clearly this amounts to $\phi^{-1}(V)$ being quasiopen for every open subset $V$ of $Y$.
\subsubsection*{\textbf{Weakly harmonic and weakly sub/super harmonic functions}} A function $u\in W_{loc}^{1,2}(X)$ is said to be {\it{weakly harmonic}} if
\begin{eqnarray*}
\int_X\langle \nabla u,\nabla \rho\rangle ~d\mu_g=0\quad\quad\ \text{for every}~\rho\in \lip_c(X).
\end{eqnarray*}
It is said to be {\it{weakly subharmonic}}, respectively weakly superharmonic, if
\begin{eqnarray*}
\int_X\langle \nabla u,\nabla \rho\rangle ~d\mu_g\leq0,~\text{resp.}~\geq0\quad\quad \text{for every}~\rho\in \lip_c(X).
\end{eqnarray*}
 A function $u\in W^{1,2}(X)$ is weakly harmonic if and only if $u$ minimizes the energy $E(v)$ among all functions $v\in W^{1,2}(X)$ such that
$v-u\in W_0^{1,2}(X)$ (see \cite{F01}). In the following we discuss on the existence of  minimizer under specific assumption on the Riemannian polyhedra.
\begin{thm}\cite{F01} Suppose the following Poincar\'{e} inequality holds
\begin{eqnarray}\label{poin}
\int_X| u|^2~d\mu_g\leq c\int_X|\nabla u|^2~d\mu_g \quad\quad\text{for all}~u\in W_0^{1,2}(X)
\end{eqnarray}
 with $c$ depending only on the admissible
Riemannian polyhedron $X$. For any $f \in W^{1,2}(X)$ the class of competing maps
\begin{eqnarray}
 W_f^{1,2}(X) = \{v\in W^{1,2}(X) : v - f \in W_0^{1,2}(X)\}
\end{eqnarray}
contains a unique weakly harmonic function $u$. That function is the unique
minimizer of  $E(u) = E_0$, where
\begin{eqnarray*}
E_0 &:=&\inf\{E(v) : v\in W^{1,2}(X), v - f \in \lip_c(X)\}\\
&=&\min\{E(v) :  v\in W_f^{1,2}(X)\}.
\end{eqnarray*}
\end{thm}
As a corollary of the above theorem we have
\begin{cor}
Assume that the domain $\Omega\subset X$ is bounded and such that the Sobolev capacity $C(X\backslash \Omega) > 0$. For any $f \in W^{1,2}(\Omega)$, the class of functions
\begin{eqnarray*}
 W_f^{1,2}(\Omega) = \{v\in W^{1,2}(\Omega) : v - f \in W_0^{1,2}(\Omega)\}
\end{eqnarray*}
has a unique solution $u$ of the equation $E(u) = E_{\Omega}$, where
\begin{eqnarray*}
E_{\Omega} :=\inf\{E(v) : v\in W^{1,2}(\Omega), v - f \in W_0^{1,2}(\Omega)\}.
\end{eqnarray*}
\end{cor}
\begin{proof}
 Since $X$ satisfies the Poincar\'{e} inequality and using  Theorem  $5.54$ in \cite{B11}, $\Omega$ satisfies the inequality (\ref{poin}). By  the above theorem, there is a unique minimizer which is weakly harmonic.
\end{proof}
After correction on a null set every weakly harmonic function on $X$ is H\"{o}lder continuous. A continuous weakly harmonic function is called {\it{harmonic}}.
\begin{remark}
From the discussion above one can see in the definition of variational capacity that there is a harmonic function $u$ which takes the minimum in
(\ref{capa}). This function is not necessarily continuous on the boundary of $\Omega\backslash E$.
\end{remark}
\subsubsection*{\textbf{Polar sets.}} A set $S\subset X$ is said to be a {\it{polar set}} for the capacity if for every pair of relatively compact open sets $U_1\Subset U_2\subset X$ such that $d(U_1,X\backslash U_2)>0$ we have
\begin{eqnarray*}
\capa (S\cap \overline{U_1},U_2)=0.
\end{eqnarray*}
According to Theorem $9.52$ in \cite{B11} (see also section $3$ in \cite{GT02}), $S$ is a polar set if and only if every point of $X$ has an open neighborhood $U$ on which there is a superharmonic function which equals  $+\infty$  at every point of $S\cap U$. An equivalent formulation is to say that $C(S)=0$.
\begin{lemme}\label{polar}
A closed set $S\subset X$ is a polar set if and only if for every
neighborhood $U$ of $S$ and every $\epsilon>0$, there exists a function $\lip(X)$ such that
 \begin{enumerate}[i.]
  \item the support of $ \varphi$ is contained in $X\backslash S$.
 \item $0\leq \varphi\leq 1$.
 \item  $ \varphi\equiv 1$ on $X\backslash U$.
\item $\int_X|\nabla  \varphi|^2<\epsilon$.
\end{enumerate}
\end{lemme}
\begin{proof}
The proof is based on the definition of polar set and followed by a similar argument  as the case of Riemannian manifolds. See Proposition $3.1$ in  \cite{T99} for the proof of the equivalence on Riemannian manifolds.
\end{proof}
\subsubsection*{\textbf{The Dirichlet space $L_{0}^{1,2}(X)$.}}  In this part we introduce the Dirichlet space $L_{0}^{1,2}(X)$ on an admissible  Riemannian polyhedron $X$ (see \cite{F01} for the definition of Dirichlet spaces). The Dirichlet space $L_{0}^{1,2}(X)$ determines a Brelot harmonic structure on $X$. Using this fact we can show, $X$ has a symmetric Green function which gives us information on the singularities of $X$.
 \begin{prop}\cite{F01}
 Suppose that, for every compact set $K \subset X$,
\begin{eqnarray}\label{cond}
\left(\int_{K}|u|~d\mu_g\right)^2\leq c(K)E(u)\quad\quad \text{for all}~x\in \lip_c(X),
\end{eqnarray}
with $c(K)$ depending only on $X$ and $K$. In particular, $X$ is non-compact. The completion $L_{0}^{1,2}(X )$ of space $\lip_c(X)$ within $L^1_{\loc}(X)$ with respect to the norm $E(u)^{1/2}$ is then a regular Dirichlet space of strongly local type. $L_{0}^{1,2}(X )$ is a subset of $W^{1,2}_{\loc}(X)$.
\end{prop}
Note that $W^{1,2}_0(X)\subset L_0^{1,2}(X)\subset W^{1,2}_{\loc}(X)$. Let
\begin{eqnarray*}
\Delta:L_0^{1,2}(X)\supset D(\Delta)\rightarrow L^2(X)
\end{eqnarray*}
denote the generator induced from $(L_0^{1,2}(X), E)$, which is a densely defined non-positive definite self-adjoint operator satisfying $E(u,v)=(\Delta u,v)_{L^2}$. Here $D(\Delta)$ denotes the domain of operator $\Delta$. We have
\begin{thm}\cite{F01}\label{Green}
 Let $(X,g,\mu_g)$ be an admissible Riemannian polyhedra such that the inequality (\ref{cond}) holds. Then $X$ has a unique symmetric Green kernel
 \begin{eqnarray*}
 G:X\times X\rightarrow (0,\infty]
 \end{eqnarray*}
 which is finite and H\"{o}lder continuous off the diagonal $X\times X$.
\end{thm}
For local questions,  condition (\ref{cond}) is not required (it is automatically satisfied with $X$ replaced by the open star of a point $a$ of $X$
relative to a sufficiently fine triangulation and in view of inequality (\ref{Poincare})). As a consequence of Theorem \ref{Green}, we have
\begin{prop}
The $(n-2)$-skeleton $X^{(n-2)}$ of an admissible Riemannian
$n$-polyhedron is a polar set.
\end{prop}
We should note that being polar is independent of the Riemannian structure on the polyhedron.
\begin{remark}\label{remov}
Every closed polar subset $F$ of $X$ is removable for Sobolev 
functions,  $W ^{1,2} (X \backslash F) = W^{ 1,2} (X )$. A larger class of removable sets in
this sense is that of all (closed) sets of $(n-1)$-dimensional Hausdorff measure
zero (see Proposition $7.7$ in  \cite{F01}).
\end{remark}
\subsection{Harmonic maps on  Riemannian polyhedra}\label{HMORM}
The energy of a map from a Riemannian domain to an arbitrary metric space was defined and investigated by Korevaar and
Schoen \cite{KS93}. Here  we give an introduction to the concept of energy of maps, energy minimizing maps and harmonic maps  on a Riemannian polyherdron. In the  case that the target $Y$ is a Riemannian $C^1$-manifold,
the energy of the map is given by the usual expression (similarly when the target is a
Riemannian polyhedron with continuous Riemannian metric).

Let $(X,g)$ be an  admissible $n$-dimensional  Riemannian polyhedron with simplexwise smooth Riemannian metric. We do not require that $g$ is continuous across lower dimensional simplexes. Let $Y$ be an arbitrary metric space. Denote by $L^2_{\loc} (X, Y)$ the
space of all $\mu_g$-measurable maps $\varphi : X\rightarrow Y$ having separable essential range \footnote{The essential range of a map $\varphi$ is a closed set of points $q\in Y$ such that for any neighborhood $V$ of $q$, $\varphi^{-1}(V)$ has positive measure.}
and for which $d_Y(\varphi(\cdot),q)\in L^2_{\loc}(X,\mu_g)$ for some point $q$ (and therefore by
the triangle inequality for any $q \in Y$). For $\varphi,\psi\in L^2_{\loc}(X,Y)$ define their distance
\begin{eqnarray*}
D(\varphi,\psi)=\left(\int_X d_Y^2(\varphi(x),\psi(x))~d\mu_g(x)\right)^{1/2}.
\end{eqnarray*}
The approximate energy density of a map $\varphi\in L^2_{\loc} (X, Y)$ is defined for $\varepsilon>0$ by
\begin{eqnarray*}
e_{\varepsilon}(\varphi)(x)=\int_{B(x,\varepsilon)}\frac{d_Y^2(\varphi(x),\varphi(x'))}{\varepsilon^{n+2}}~d\mu_g(x').
\end{eqnarray*}
The function $e_{\varepsilon}(\varphi)$ is of class $L^1_{\loc}(X,\mu_g)$ (see \cite{KS93}).
\begin{defin}
The energy $E(\varphi)$ of a map $\varphi$ of class $L^2_{\loc} (X, Y)$ is defined as
\begin{eqnarray*}
E(\varphi)=\sup_{f\in C_c(X,[0,1])}\left(\limsup_{\varepsilon\rightarrow 0}\int_X fe_{\varepsilon}(\varphi)~d\mu_g\right).
\end{eqnarray*}
\end{defin}
We say that $\varphi$ is locally of finite energy and write $\varphi\in W^{1,2}_{\loc}(X,Y)$, if $E(\varphi|_U)<\infty$ for every relatively compact domain $U \subset X$. For example every Lipschitz continuous map $\varphi:X\rightarrow Y$ is in $ W^{1,2}_{\loc}(X,Y)$. Now we give a necessary and sufficient condition for a map $\varphi$ to be in $W^{1,2}_{\loc}(X,Y)$.
\begin{lemme}
Let $(X, g)$ be an admissible $n$-dimensional Riemannian polyhedron with simplexwise smooth Riemannian metric and $(Y, d_Y)$ a metric
space. A map $\varphi\in L^2_{\loc} (X, Y)$ is locally of finite energy if and only if there is a function
$e(\varphi)\in L^1_{\loc} (X)$  such that $e_{\varepsilon}(\varphi)\rightarrow e(\varphi)$ as $\varepsilon\rightarrow 0$, in the sense of weak
convergence of measures
\begin{eqnarray*}
\lim_{\varepsilon\rightarrow 0}\int_X fe_{\varepsilon}(\varphi)~d\mu_g=\int_X fe(\varphi)~d\mu_g\quad\quad f\in C_c(X).
\end{eqnarray*}
\end{lemme}
\subsubsection*{\textbf{Energy of maps into Riemannian manifolds.}} Let the domain be an arbitrary admissible Riemannian polyhedron
$(X, g)$ ($g$ is only measurable with local elliptic bounds, unless  otherwise specified) and the target is a Riemannian $C^1$-manifold $(N, h)$ without boundary, $X$ of dimension $n$ and  $Y$ of dimension $m$.

A chart $\eta$ of $N$, $\eta : V \rightarrow \mathbb R^m$ is  bi-Lipschitz if the components $h_{\alpha\beta}$ of $h|_V$
have elliptic bounds
\begin{eqnarray}
\Lambda_V^{-2} \sum_{\alpha=1}^{m}(\eta^{\alpha})^2 \leq h_{\alpha\beta}\eta^{\alpha}\eta^{\beta} \leq \Lambda_V^{2} \sum_{\alpha=1}^{m}(\eta^{\alpha})^2.
\end{eqnarray}
Relative to a given countable atlas on a Riemannian $C^1$-manifold $(N, h)$,
 a map $\varphi : (X, g)\rightarrow (N, h)$ is of class $W^{1,2}_{\loc}(X, N)$, or locally
of finite energy, if
\begin{enumerate}[i.]
\item $\varphi$ is a quasicontinuous (after correction on a set of measure zero).
\item  Its components $\varphi_1,\ldots,\varphi_m$ in charts $\eta : V \rightarrow \mathbb R^m$ are of class $W ^{1,2}(U)$
for every quasiopen set $U\subset \varphi^{-1}(V)$ of compact closure in $X$.
\item The energy density $e(\varphi)$ of $\varphi$, defined a.e. in each of the quasiopen
sets $ \varphi^{-1}(V)$ covering $X$ by
\begin{eqnarray*}
e(\varphi)=(h_{\alpha\beta}\circ\varphi)g( \nabla\varphi^{\alpha},\nabla\varphi^{\beta})
\end{eqnarray*}
is locally integrable over $(X, \mu_g)$.
\end{enumerate}
The energy of $\varphi \in W^{1,2}_{\loc}(X, N)$ is defined by $E(\varphi) = \int_X e(\varphi)~d\mu_g$.

\subsubsection*{\textbf{Energy minimizing maps.}} We suppose that $(X, g)$, $n$-dimensional admissible Riemannian polyhedra
with $g$ simplexwise smooth and $Y$ any metric space.  A map $\varphi \in W^{1,2}_{\loc}(X,Y)$ is
said to be locally energy minimizing if $X$ can be covered by relatively compact
domains $ U \subset X$ for which $E(\phi|_U)\leq E(\psi|_U)$ for every map $\psi \in W^{1,2}_{\loc}(X,Y)$
such that $\varphi=\psi$ a.e. in $X \backslash U$.
\subsubsection*{\textbf{Harmonic maps.}}
Consider an admissible Riemannian polyhedron $(X, g)$, of dimension $n$, and a
metric space $(Y,d_Y)$,
\begin{defin}\label{harmonicmap}
 A harmonic map $\varphi :X\rightarrow Y$ is a continuous map of
class $\varphi \in W^{1,2}_{\loc}(X,Y)$, which is locally energy minimizing in the sense that $X$ can
be covered by relatively compact subdomains $U$, for each of which there is
an open set $ V\supset\varphi(U)$ in $Y$ such that
\begin{eqnarray*}
E(\varphi|_U) \leq E(\psi|_U)
\end{eqnarray*}
for every continuous map $\psi \in W^{1,2}_{\loc}(X,Y)$ with $\psi(U) \subset V$ and $\varphi=\psi$ in $X \backslash U$.
\end{defin}
Every continuous, locally energy minimizing map $\varphi :X\rightarrow Y$ is harmonic. Also if
$Y$ is a simply connected complete Riemannian polyhedron
of non-positive curvature, then a harmonic map $\varphi :X\rightarrow Y$ is the same as
a continuous locally energy minimizing map.

  We proceed now to characterize harmonic maps are continuous, weakly harmonic maps. We consider here $(X, g)$ to be an arbitrary admissible Riemannian polyhedron and $g$ just bounded measurable with local elliptic bounds, $X$ of dimension $n$, and
$(N, h)$  a smooth Riemannian manifold without boundary, and the dimension of $N$ is $m$. We denote by $\Gamma_{\alpha\beta}^k$ the Christoffel symbols on $N$.
\begin{defin}\label{whar}
A weakly harmonic map $\varphi:X\rightarrow N$ is a quasicontinuous
map of class $ W^{1,2}_{\loc}(X,N)$ with the following property:
for any chart $\eta: V \rightarrow \mathbb R^n$ on $N$ and any quasiopen set $U \subset \varphi^{-1}(V)$ of
compact closure in $X$, the equation
\begin{equation}\label{EL2}
\int_U \langle\nabla\lambda,\nabla \varphi^k\rangle ~d\mu_g=\int_U\lambda\cdot (\Gamma_{\alpha\beta}^k\circ\varphi)\langle\nabla\varphi^{\alpha},\nabla\varphi^{\beta}\rangle~d\mu_g
\end{equation}
holds for every $k = 1,\ldots, m$ and every bounded function $\lambda\in W^{1,2}_{0}(U) $.
\end{defin}
According to \cite{F01}, a continuous map $ \varphi\in W^{1,2}_{\loc}(X,N)$ is harmonic (Definition \ref{harmonicmap}) if and only if it is weakly harmonic (Definition  \ref{whar}).
%
%
%
\section{Ricci Curvature on Riemannian Polyhedra}\label{Ricci}
 In the past few years, several notions of boundedness of Ricci curvature from below  on general metric spaces have appeared. Sturm \cite{S06} and Lott-Villani \cite{V09} independently introduced the so called curvature-dimension condition on a metric measure space denoted by $\CD(K,N)$.
 The curvature dimension condition implies the generalized Brunn-Minkowski inequality (hence the Bishop-Gromov comparison and Bonnet-Myer's theorem) and a Poincar\'e inequality (see \cite{S06,LV07,V09}). Meanwhile, Sturm and Ohta introduced a measure contraction property
 denoted as $\MCP(K,N)$ in Ohta's work. The condition $\MCP(K,N)$ also implies the Bishop-Gromov comparison, Bonnet-Myer's theorem and a Poincar\'e inequality (see \cite{S06,O07}).
 Note that all of these generalized notions of Ricci curvature bounded below are equivalent to the classical one on smooth Riemannian manifolds.  Then after the reduced curvature dimension condition  $\CD^*(K,N)$ has been introduce by Bacher and Sturm in \cite{BS10}
 to overcome local-to-global property and it is equivalent to local $\CD(K,N)$ condition.  More recently  the notion of  Riemannian curvature dimension condition $\RCD(K,\infty)$
  has been introduced in \cite{AGS12} and it is equivalent to the $\CD(K,\infty)$ on infinitesimally Hilbertian metric measure spaces. The finite dimensional version of this notion
 $\RCD(K,N)$ has been introduced in \cite{EKS13} and independently in \cite{AMS13} and it is equivalent to $\CD^*(K,N)$ on infinitesimally Hilbertian metric measure spaces. 
 
 Here we define both $\CD(K,N)$, $\MCP(K,N)$ and show that on a Riemannian polyhedron we can use both of them.
 In the following definitions, we always assume that  $(X,d)$  is a separable length space, $P(X)$ is the set of all Borel probability measures $\mu$ satisfying  $\int_X d_X(x,y)^2 ~d\mu(y)<\infty$ for some $x\in X$. $P_2(X)$ is the set $P(X)$ equipped with the $L^2$-Wasserstein distance $W_2$ defined as
\begin{eqnarray*}
W_2(\mu_0,\mu_1)^2=\inf_{\pi}\int_{X\times X}d(x_0,x_1)^2 d\pi(x_0,x_1)
\end{eqnarray*}
for $\mu_0$, $\mu_1$ in $P_2(X)$ and $\pi$ in $P(X\times X)$ ranges between all transference plan between $\mu_0$ and $\mu_1$. A transference plan is defined as 
\begin{eqnarray*}
{p_0}_*(\pi)=\mu_0\quad\quad{p_1}_*(\pi)=\mu_1
\end{eqnarray*}
where $p_0,p_1:X\times X\rightarrow X$ are projection to the first and second factors respectively.
\subsubsection*{\textbf{Curvature Dimension Condition.}}  We now define the notion of spaces satisfying $\CD(K,N)$ condition following \cite{V09}. Suppose $(X,d)$ is a compact length space. Let $U : [0,\infty)\rightarrow \mathbb R$ be a continuous convex function with $U(0) = 0$. We define the non-negative function
\begin{eqnarray*}
p(r)=rU'_+(r)-U(r)
\end{eqnarray*}
with $p(0)=0$.  Given a reference probability measure $\nu\in P_2(X)$, define the function $U_{\nu}:P_2(X)\rightarrow \mathbb R\cup\{\infty\}$
by
\begin{eqnarray*}
U_{\nu}(\mu)=\int_X U(\rho(x))d\nu(x)+U'(\infty)\mu_s(X)
\end{eqnarray*}
where
$$
\mu=\rho\nu+\mu_s
$$
is the Lebesgue decomposition of $\mu$ with respect to $\nu$ into an absolutely continuous
part $\rho\nu$ and a singular part $\mu_s$, and
\begin{eqnarray*}
U'(\infty)=\lim_{r\rightarrow \infty}\frac{U(r)}{r}.
\end{eqnarray*}
If $N \in [1,\infty)$ then we define $\cal{DC}_N$ to be the set of such functions $U$ so
that
\begin{eqnarray*}
\psi(\lambda)=\lambda^NU(\lambda^{-N})
\end{eqnarray*}
is convex on $(0,\infty)$. We further define $\cal{DC}_{\infty}$ to be the set of such functions $U$
so that the function
\begin{eqnarray*}
\psi(\lambda)=e^{\lambda}U(e^{-\lambda})
\end{eqnarray*}
is convex on $(-\infty,\infty)$. A relevant example of an element in $\cal{DC}_N$ is given by
\begin{align}
H_{N,\nu}=\left\{\begin{array}{ll}
Nr(1-r^{-1\slash N})& \text{if}~1<N<\infty,\\
r\log r& \text{if}~N=\infty.
\end{array} \right.
\end{align}
\begin{defin}
\begin{enumerate}[i.]
\item Given $N \in [1,\infty]$, we say that a compact measured length
space $(X, d, \nu)$ has non-negative $N$-Ricci curvature if for all $\mu_0, \mu_1 \in P_2(X)$
with $\supp(\mu_0) \subset \supp(\nu)$ and $\supp(\mu_1) \subset \supp(\nu)$, there is some Wasserstein
geodesic $\{\mu_t\}_{t\in[0,1]}$ from $\mu_0$ to $\mu_1$ so that for all $U \in \cal{DC}_N$ and all $t \in [0, 1]$,
\begin{eqnarray}\label{DCN}
U_{\nu}(\mu_t)\leq t U_{\nu}(\mu_1)+ (1-t)U_{\nu}(\mu_0).
\end{eqnarray}
\item Given $K \in \mathbb R$, we say that $(X, d, \nu)$ has $\infty$-Ricci curvature bounded below by
$K$ if for all $\mu_0, \mu_1 \in P_2(X)$ with $\supp(\mu_0) \subset \supp(\nu)$ and $\supp(\mu_1) \subset \supp(\nu)$,
there is some Wasserstein geodesic $\{\mu_t\}_{t\in[0,1]}$ from $\mu_0$ to $\mu_1$ so that for all
 $U \in \cal{DC}_{\infty}$ and all $t \in [0, 1]$,
\begin{eqnarray}\label{DCinf}
U_{\nu}(\mu_t)\leq t U_{\nu}(\mu_1)+ (1-t)U_{\nu}(\mu_0)-\frac{1}{2}\lambda(U)t(1-t)W_2(\mu_0,\mu_1)^2
\end{eqnarray}
where $\lambda:\cal{DC}_{\infty}\rightarrow \mathbb R\cup\{-\infty\}$ is defined as
\begin{eqnarray*}
\lambda(U)=\inf_{r>0}K\frac{p(r)}{r}=\left\{\begin{array}{ll}
K\lim_{r\rightarrow0^+}\frac{p(r)}{r}&\text{if}~K>0,\\
0&\text{if}~K=0,\\
K\lim_{r\rightarrow\infty}\frac{p(r)}{r}&\text{if}~K<0.
\end{array} \right.
\end{eqnarray*}
\end{enumerate}
\end{defin}
Note that inequalities (\ref{DCN}) and (\ref{DCinf}) are only assumed to hold along
some Wasserstein geodesic from $\mu_0$ to $\mu_1$, and not necessarily along all such
geodesics. This is what is called weak displacement convexity.
%
\begin{prop}\label{bishop1}
If a compact measured length space $(X, d, \nu)$ has non-negative $N$-Ricci curvature for some $N\in [1,\infty)$, then for all $x\in \supp(\nu)$ and all $0 < r_1\leq r_2$ the following inequality holds
\begin{eqnarray*}
 \nu(B_{r_2}(x))\leq \left(\frac{r_2}{r_1}\right)^N\nu(B_{r_1}(x)).
\end{eqnarray*}
\end{prop}

 To generalize the notion of $N$-Ricci curvature to the non-compact case, we always consider a complete pointed locally compact metric measure space $(X,\star,\nu)$. Also for $U_{\nu}$ to be a  well-defined functional on $P_2(X)$, we impose the restriction $\nu \in M_{-2(N-1)}$, where $M_{-2(N-1)}$
  is the space of all non-negative Radon measures $\nu$ on $X$ such that
 \begin{eqnarray*}
  \int_X(1+d(\star,x)^2)^{-(N-1)}~d\nu(x)<\infty.
  \end{eqnarray*}
  We define $M_{-\infty}$ by the condition $\int_X e^{-cd(\star,x)^2}~d\nu(x)<\infty$, where $c$ is a fixed positive constant. We should mention that most of the results for compact case (for example the Bishop-Gromov comparison)
 are valid for the  non-compact case.
\subsubsection*{\textbf{Measure Contraction Property}} We define now the notion of measure contraction property $\MCP(K,N)$ following \cite{O07}.  Let $(X, d_X)$ be a length space and  $\mu$  a Borel measure on $X$ such that $0 < \mu(B(x, r)) < \infty$  for every $x\in X$
and $r > 0$, where $B(x, r)$  denotes the open ball with center $x \in X$ and
radius $r > 0$.

Let $\Gamma$ be the set of minimal geodesics  $\gamma : [0, 1]\rightarrow X$  and define
the evaluation map $e_t$  by $e_t(\gamma) := \gamma(t)$ for each $t\in [0, 1]$. We regard $\Gamma$ as a
subset of the set of Lipschitz maps $\lip([0, 1],X)$ with the uniform topology.  A dynamical
transference plan $\Pi$ is a Borel probability measure on $\Gamma$ and a path $\{\mu_t\}_{t\in[0,1]}\subset P_2(X)$
given by $\mu_t=(e_t)_*\Pi$ is called a displacement interpolation associated to $\Pi$. For the exact definition of dynamical transference plan and displacement interpolation we refer the reader to \cite{V09}.
For $K\in\mathbb R$, we define the function $s_K$ on $[0,\infty)$ (on $[0,\pi\slash \sqrt {K})$ if $K>0$) by
\begin{align}
s_K(t):=\left\{\begin{array}{ll}
(1\slash \sqrt{K}) \sin(\sqrt{K}t)&\text{if}~K>0,\\
 t &\text{if}~K=0,\\
(1\slash \sqrt{-K})\sinh(\sqrt{-K}t)&\text{if}~K<0.
\end{array} \right.
\end{align}
\begin{defin}
For $K,N\in\mathbb R$ with $N > 1$, or with $ K\leq0$ and $N = 1$, a metric measure
space $(X,d,\mu)$ is said to satisfy the $(K,N)$-measure contraction property $\MCP(K,N)$, if for every point $x\in X$ and measurable set $A \subset  X$ (provided that
$A\subset B(x,\pi\sqrt{(N-1)\slash K })$ if $K > 0$) with $0 < \mu(A) <\infty$, there exists a displacement
interpolation $\{\mu_t\}_{t\in[0,1]}\subset P_2(X)$ associated to a dynamical transference plan $\Pi = \Pi_{x,A}$
satisfying
\begin{enumerate}[i.]
\item  We have $\mu_0=\delta_x$ and $\mu_1 =(\mu|_A)^{-}$ as measures where we denote by $(\mu|_A)^-$ the
normalization of $\mu|_A$, i.e. $(\mu|_A)^- := \mu(A)^{-1}\cdot \mu|_A$.\\
\item  For every $t\in [0, 1]$,
\begin{eqnarray*}
d\mu\geq (e_t)_*\left( t\left\{ \frac{s_K(t ~l(\gamma)\slash\sqrt{N-1})}{s_K( l(\gamma)\slash\sqrt{N-1})}\right\}^{N-1}\mu(A)d\Pi(\gamma)\right)
\end{eqnarray*}
holds as measures on $X$, where we set $0\slash0 = 1$ and by convention, we read
\begin{eqnarray*}
\left\{ \frac{s_K(t ~l(\gamma)\slash\sqrt{N-1})}{s_K( l(\gamma)\slash\sqrt{N-1})}\right\}^{N-1}=1
\end{eqnarray*}
if $K\leq 0$ and $N=1$.
\end{enumerate}
\end{defin}
Here we state two results that we are going to use in the sequel.
\begin{prop}
Let $(M, g)$ be an $n$-dimensional, complete Riemannian manifold without
boundary with $n \geq 2$. Then a metric measure space $(M, d_g, \nu_g)$ satisfies the $\MCP(K, n)$
if and only if $\ric_g \geq K$ holds. Here $d_g$ and $\nu_g$ denote the Riemannian distance and Riemannian volume element.
\end{prop}
In the following theorem we state Bishop volume comparison theorem for the space satisfying $\MCP(K,N)$.
\begin{prop}\label{bishop2}Let $(X, \mu)$ be a metric space
satisfying the $\MCP(K,N)$. Then for any $x\in X$, the function
\begin{eqnarray*}
\mu(B(x, r))\cdot\left\{\int_0^r s_K\left(\frac{s}{\sqrt{N-1}}\right)^{N-1}d_s\right\}^{-1}
\end{eqnarray*}
is monotone non-increasing in $r \in (0,\infty)$ ($r\in (0,\pi\sqrt{\tfrac{N-1}{K}})$
 if $K > 0$).
\end{prop}
In the following we show that we can apply both  measure contraction property and curvature dimension condition to a complete Riemannian polyhedra  $(X,g,\mu_g)$. By previous section, a Riemannian polyhedron $(X,g,\mu_g)$ with the metric $d_X=d_X^g$ is a length space.  Also for any $x,y\in X$ we have
\begin{eqnarray*}
e(x,y)\leq d^e_X(x,y)
\end{eqnarray*}
where $e(x,y)$ denotes the Euclidian distance. It is easy  then to show that $\mu_g$ is in $M_{-2(N-1)}$ and so on a complete Riemannian polyhedron we can use the notion of  $\CD(K,N)$. Also $\mu_g$ is Borel and by Lemma $4.4$ in \cite{F01}, for any $r$ there exist a constant $c(r)$ such that
\begin{eqnarray*}
c(r)^{-1}\Lambda^{-2n}r^n\leq\mu_g(B(x,r))\leq c(r)\Lambda^{2n}r^n
\end{eqnarray*}
for all $x\in X$. Therefore $0<\mu_g(B(x,r))<\infty$ and the notion of $\MCP(K,N)$ is also applicable 
for $N\geq n$. By Theorem $2.4.3$ in \cite{AT04}, we have the Hausdorff dimension is $n$ and by Corollary $2.7$ in \cite{O07} $N$ should be greater than $n$.\\

In the rest of this work by $\ric_{N,\mu_g}\geq K$ can be taken to mean that $(X,g,\mu_g)$ satisfies either $\MCP(K,N)$ or $\CD(K,N)$ unless otherwise specified. In the following Lemma we show that any complete Riemannian polyhedron with non-negative Ricci curvature has infinite volume.
\begin{lemme}\label{bishop}
Let $(X,g,\mu_g)$ be a complete, non-compact Riemannian polyhedron of dimension $n$. If $\ric_{N,\mu_g}(X)\geq0$ for $N\geq n$, then $X$ has infinite volume.
\end{lemme}
\begin{proof}
First we consider the case $\MCP(0,N)$. By the Bishop comparison theorem, Theorem \ref{bishop2},   for $x\in X$ and all $0<r_1\leq r_2$,
\begin{eqnarray*}
\mu_g(B_{r_2}(x))\leq \left( \frac{r_2}{r_1}\right)^N \mu_g(B_{r_1}(x)).
\end{eqnarray*}
By Proposition $10.1.1$  in \cite{P05} for every point in $X$, there exist a geodesic ray from that point. Consider a  geodesic ray $\gamma(t)$, $0\leq t<\infty$, such that $\gamma(0)=x$. We construct the balls $B(\gamma(t),t-1)$ and $B(\gamma(t),t+1)$ centered at $\gamma(t)$ with radius $t-1$ and $t+1$. We have
\begin{eqnarray*}
\frac{\mu_g(B(\gamma(0),1))+\mu_g(B(\gamma(t),t-1))}{\mu_g(B(\gamma(t),t-1))}\leq\frac{\mu_g(B(\gamma(t),t+1))}{\mu_g(B(\gamma(t),t-1))}\leq\left( \frac{t+1}{t-1}\right)^N,
\end{eqnarray*}
and so
\begin{eqnarray*}
 1+\frac{\mu_g(B(\gamma(0),1))}{\mu_g(B(\gamma(t),t-1))}\leq\left( \frac{t+1}{t-1}\right)^N.
\end{eqnarray*}
Letting $t\rightarrow \infty $, we get $\mu_g(B(\gamma(t),t-1))\rightarrow \infty$ and therefore $X$ has infinite volume.
By Theorem \ref{bishop1}, and since $X$ is a complete locally compact length space, we can repeat the proof for the case when $X$  satisfies the non-negative $N$-Ricci curvature condition $\CD(0,N)$, for $N\in(1,\infty)$.
\end{proof}
Here we recall some remarks that we need through the rest of this paper.
\begin{remark}\label{ric-man}
\begin{enumerate}[i.]
\item By Remark $5.8$ in \cite{S06} if $(X,d,\mu)$ satisfy $\MCP(K,N)$ so does  any  convex set $A\subset X$. When $X$ is a smooth pseudomanifold, for any  point $x\in X\backslash S$, there exist a  closed totally convex neighborhood $V$ around $x$ (for every  point in a Riemannian manifold there is a geodesic ball which is totally convex).  Therefore if $X$ satisfies  $\ric_{N,\mu_g}\geq K$, so does $X\backslash S$. The same result is valid on metric measure spaces with $\CD(K,N)$ condition, see Theorem 5.53 in \cite{V09}.
\item  By definition of $\CD^*(K,N)$, $\CD^*(0,N)$ is equivalent to $\CD(0,N)$.
\item Since the Sobolov space $W^{1,2}(X)$ on an admissible Riemannian polyhedra $(X,g,\mu_g)$ is a Hilbert space, then $X$ is infinitesimally Hilbertian\footnote{see Definition $4.18$ in \cite{G12} for the definition of infinitesimally Hilbertian.}. Therefore the notion $\RCD(K,\infty)$ is equivalent to $\CD(K,\infty)$ and $\RCD(0,N)$ is equivalent to $\CD(0,N)$.
\end{enumerate}
\end{remark}

\section{Some Function Theoretic Properties  On Complete Riemannain Polyhedra}\label{complete}
\subsection{Liouville-type Theorems for Functions}\label{complete1}
The aim of this section is to generalize some of the results in \cite{Y82} in order to prove some vanishing theorems for harmonic maps on Riemannian polyhedra.
In \cite{Y82}, Yau used the Gaffney's Stokes theorem on complete Riemannian manifolds to prove that every smooth subharmonic function with bounded  $\|\nabla f\|_{L_1}$ is harmonic. He uses this fact to prove that there is no non-constant $L^p$, $p>1$, non-negative subharmonic function on a complete
manifold. We generalize this theorem to complete admissible polyhedra for $p=2$ as stated in Theorem \ref{yau1}.
\begin{proof}[Proof of Theorem \ref{yau1}]
Fix a base point $x_0\in X$ and define $\rho:X\rightarrow \mathbb R$  as
\begin{eqnarray*}
\rho(x)=\max \{0, \min\{1,2-\frac{1}{R}d(x,x_0)\}\}
\end{eqnarray*}
Observe that $\rho$ is $\frac{1}{R}$-Lipschitz and $\rho=0$ on $X\backslash B(x_0,2R)$ and $\rho=1$ on $B(x_0,R)$.  Since $f$ is subharmonic,
\begin{eqnarray*}
0&\geq&\int_X\langle \nabla(\rho^2 f),\nabla f\rangle ~d\mu_g\\
&=&\int_X\langle (\nabla \rho^2) f+ (\nabla f)\rho^2, \nabla f\rangle~ d\mu_g \\
&=&\frac{1}{2}\int_X\langle \nabla \rho^2,\nabla f^2\rangle ~d\mu_g+\int_X\rho^2|\nabla f|^2~d\mu_g\\
&=&2\int_X\langle \rho\nabla \rho,f\nabla f\rangle ~d\mu_g+\int_X\rho^2|\nabla f|^2~d\mu_g.
\end{eqnarray*}
From Cauchy-Schwarz we have 
\begin{eqnarray*}
\int_X\langle \rho\nabla \rho,f\nabla f\rangle~d\mu_g&=&\int_X\langle f\nabla \rho,\rho\nabla f\rangle ~d\mu_g\\
&\geq& -{\left(\int_X  |f\nabla \rho|^2 ~d\mu_g\right)}^{\frac{1}{2}}{\left( \int_X|\rho\nabla f|^2~d\mu_g\right)}^{\frac{1}{2}}.
\end{eqnarray*}
Combining the two previous inequalities, we obtain
\begin{eqnarray*}
0&\geq& 2\int_X\langle \rho\nabla \rho,f\nabla f\rangle d\mu_g+\int_X\rho^2|\nabla f|^2~d\mu_g\\
&\geq& \int_{B_{2R}\backslash B_R}|\rho\nabla f|^2~d\mu_g- 2{\left (\int_{B_{2R}\backslash B_R}|f\nabla \rho|^2~d\mu_g\right)}^{\frac{1}{2}}{\left(\int_{B_{2R}\backslash B_R}|\rho\nabla f|^2~d\mu_g\right)}^{\frac{1}{2}}\\
&+&\int_{B_{R}}|\nabla f|^2d~\mu_g.
\end{eqnarray*}
The last line is a polynomial,  $P(\psi)=\psi^2-2b\psi+c$, where $\psi$  is
 \begin{eqnarray*}
 {\left(\int_{B_{2R}\backslash B_R}|\rho\nabla f|^2~d\mu_g\right)}^{\frac{1}{2}}
 \end{eqnarray*}
and it has non-positive value only if $b^2\geq c$, which means that
\begin{eqnarray*}
\int_{B_{R}}|\nabla f|^2d\mu_g\leq \int_{B_{2R}\backslash B_R}f^2|\nabla\rho|^2\leq \frac{c^2}{R^2}\int_{B_{2R}}f^2~d\mu_g
\end{eqnarray*}
and so
\begin{eqnarray}\label{polyin}
\int_{B_{R}}|\nabla f|^2d\mu_g \leq \frac{c^2}{R^2}\int_{X}f^2~d\mu_g.
\end{eqnarray}
Sending $R$ to infinity and using the fact that $f$ has finite $L^2$-norm, we conclude that
\begin{eqnarray*}
\int_{X}|\nabla f|^2~d\mu_g =0.
\end{eqnarray*}
 Since $X$ is admissible, $f$ is constant on $X$; first we prove that $f$ is constant on each maximal $n$-simplex $S$ and then using the $n-1$-chainability of $X$, we prove this in the star of any vertex $p$ of $X$ and then by connectedness on $X$.
 \end{proof}
In the following theorem, we show that the Laplacian of a weakly subharmonic function $f\in W_{\loc}^{1,2}(X)$ on a pseudomanifold in the distributional sense is  a locally finite Borel measure. This gives us a verifying of Green's formula on these spaces. We then use this theorem, to prove that a continuous weakly subharmonic function with   $\|\nabla f\|_{L^1}<\infty$ on a complete normal circuit is harmonic.
\begin{thm}\label{radon}
Let $(X,g,\mu_g)$ be an $n$-pseudomanifold. Let $f$ be a weakly subharmonic function in $W_{\loc}^{1,2}(X)$, such that $\|\nabla f\|_{L^1}$ is finite. Then there exists a unique locally finite Borel measure $\textbf{m}_f$ on $X$ such that
\begin{eqnarray*}
\int_X h ~\textbf{m}_f=-\int_X \langle\nabla f, \nabla h \rangle~d\mu_g\quad\quad\text{for all}~h\in \lip_{c}(X).
\end{eqnarray*}
\end{thm}
\begin{proof}
We consider the Lipschitz manifold $M=X\backslash S$ and the chart $\{(U_{\alpha},\psi_{\alpha})\}$ on $M$. We show that
\begin{eqnarray*}
\Lambda _{\alpha}(h)=-\int_{U_{\alpha}}\langle \nabla f, \nabla h\circ\psi_{\alpha} \rangle~d\mu_g
\end{eqnarray*}
is a linear continuous functional on $D_{\alpha}=\lip_c(\psi_{\alpha}(U_{\alpha}))$ with respect to the topology of uniform convergence on compact sets. The linearity is obvious. We have
\begin{eqnarray*}
\Lambda _{\alpha}(h)=-\int_{U_{\alpha}}\langle \nabla f, \nabla h\circ\psi_{\alpha} \rangle~d\mu_g\leq \sup_{x\in U_{\alpha}}|\nabla h(x)|\cdot \|\nabla f\|_{L^1(U_{\alpha})}
\end{eqnarray*}
and so $\Lambda_{\alpha}$ is continuous. Since $\lip_c(U)$ is dense in $C_c(U)$ for a locally compact domain $U$, see Proposition $1.11$ in \cite{B11}, then $\Lambda_{\alpha}$ is also continuous on $C_c(\psi_{\alpha}(U_{\alpha}))$.  By assumption $f$ is subharmonic and so $\Lambda_{\alpha}$ is positive. By Riesz representation theorem,  $\Lambda_{\alpha}$ is a unique positive Radon measure. It follows that there is a positive Radon measure $m_{\alpha}$ such that
\begin{eqnarray*}
\Lambda_{\alpha}(h)=\int_{U_{\alpha}}h~dm_{\alpha}.
\end{eqnarray*}
Now we consider the partition of unity $\{\rho_{\alpha}\}$ subordinate to $\{U_{\alpha}\}$. We put $m=\sum_{\alpha}\rho_{\alpha}\psi_*(m_{\alpha})$ and we define $\textbf{m}_f(U)=m(U\backslash S)$ on each Borel set $U$. Obviously $\textbf{m}_f$ is positive and locally finite. The uniqueness comes from the uniqueness of $m_{\alpha}$.
\end{proof}
We recall a remark concerning the above theorem.
\begin{remark}
 Gigli introduced the notion of Laplacian as a set of locally finite Borel measure (see Definition $4.4$ in \cite{G12}). There he proved that on infinitesimally Hilbertian spaces this set contains only one element. 
\end{remark}
 In the smooth setting, as a corollary of Gaffney's Stokes theorem, we have that on a complete Riemannian manifold every smooth subharmonic function $f$ with bounded $\|\nabla f\|_{L^1}$ is harmonic. We generalize this theorem as follows:
\begin{thm}\label{sub-har}
Let $(X,g,\mu_g)$ be a complete non-compact pseudomanifold. Let $f$  be  continuous, weakly subharmonic and  belonging to $W^{1,2}_{\loc}(X)$ such that $\|\nabla f\|_{L^1}$ is finite. Then $f$ is a harmonic function.
\end{thm}
\begin{proof}
We put $A_1=\|\nabla f\|_{L^1}$. We consider a sequence of cut-off functions $\rho_n$ for fixed $q\in X$ such that $\rho_n$ is $\frac{1}{n}$-Lipschitz and such that $\rho_n$ is equal to $1$ on $B(q,R)$  and its support is in $B(q,R+n)$. $f$ is a subharmonic function which satisfies the condition of previous lemma, so there is a unique Borel measure $\textbf{m}_f$ such that
\begin{eqnarray*}
0\leq\int_X\rho_n~d\textbf{m}_f=-\int_X\langle\nabla \rho_n, \nabla f\rangle ~d\mu_g\leq \int_X|\nabla\rho_n||\nabla f|~d\mu_g\leq\frac{1}{n}A_1
\end{eqnarray*}
and
\begin{eqnarray*}
0\leq\int_{B(q,R)}d\textbf{m}_f\leq\int_X\rho_n~d\textbf{m}_f\leq\frac{1}{n}A_1.
\end{eqnarray*}
Let $h$ be any function in $\lip_c(X)$ with support in $B(q,R)$. We have
\begin{eqnarray*}
0\leq\int_X h~d\textbf{m}_f\leq (\sup_X h)\frac{1}{n}A_1
\end{eqnarray*}
and  tending $n$ to infinity, we have
\begin{eqnarray*}
\int_X h~d\textbf{m}_f=-\int_X\langle\nabla h, \nabla f\rangle ~d\mu_g=0
\end{eqnarray*}
 and implying that $f$ is harmonic.
\end{proof}
Now we prove a generalization of Proposition $2$ in \cite{Y82}, see Theorem \ref{pseudo} for the exact statement. We give here another proof of the theorem  above for smooth pseudomanifolds under the extra assumption that $f$ should have finite energy. Instead of Theorem \ref{radon}, we goal Cheeger's Green formula on compact smooth pseudomanifolds in the proof.
\begin{proof}[Proof of Theorem \ref{pseudo}]
We put $A_1=\|\nabla f\|_{L^1}$ and $A_2=\|\nabla f\|_{L^2}$. We present the proof in several steps.
\begin{step}\normalfont
 We consider a sequence of cut-off functions $\rho_n$  as above such that the support of $\rho_n$ is in $B(q,R+n)$ for fixed $q\in X\backslash X^{n-2}$ and some $R$ and $\rho_n$ is equal to $1$ on $B(q,R)$ and $\rho_n$ is $\frac{1}{n}$-Lipschitz.
\end{step}
\begin{step}\normalfont
The $(n-2)$-skeleton in $X$, $X^{n-2}$, is a polar set. We consider  shrinking bounded neighborhoods  $U_j$ of  $X^{n-2}$ in $B(q,R+j)$, such that in each $B(q,R+j)$ we have
\begin{eqnarray*}
U_j\supset U_{j+1}\supset\ldots \supset\bigcap_{k=1}^{\infty}U_j.
\end{eqnarray*}
By the definition of polar set, for the open domains $U_j$ and $U_{j-1}$, we have $\capa(X^{n-2}\cap U_j,U_{j-1})=0$. This means that for every $j$, there exists a function $\varphi_j \in \lip(X)$ such that $\varphi_j\equiv 1$ in a neighborhood of
$X^{n-2}\cap U_j$ and $\varphi_j$ is zero outside $U_{j-1}$ and $\int_X|\nabla  \varphi_j|^2<\frac{1}{j}$. Moreover we have $0\leq\varphi_j\leq1$.

 We put $\eta_j=1-\varphi_j$. The function $\eta_j$ has the property that the closure of its support, $\overline{\supp {\eta_j}}$, is contained in $X\backslash X^{n-2}$ and the set $K_j=\overline{\supp {\eta_j}}\cap\overline{B(q,R+j)}$ is  compact. Furthermore $K_j$ make an exhaustion of $M=X\backslash X^{n-2}$.
\end{step}
\begin{step}\normalfont
 According to Theorem $2$  in \cite{GW79}, for any $j$, there exist a smooth subharmonic function $f_j$ on $M$ such that $\sup_{x\in K_j}|f_j(x)-f(x)|<\frac{1}{j}$ and $|\nabla f_{j}(x)|\leq |\nabla f(x)|$ on $K_j$.
\end{step}
\begin{step}\normalfont In this step we prove
 \begin{eqnarray*}
  \int_{M}\Delta f_j\cdot\xi_j~d\mu_{g}&=&-\int_{M}\langle\nabla f_j,\nabla\xi_j\rangle~d\mu_{g}.
 \end{eqnarray*}
 where $\xi_j=\rho_j\cdot\eta_j$. To prove the above equality, first we recall a Remark from  \cite{C79}.
\begin{remark}
Let $(Y,h)$ be a closed $n$-dimensional admissible Riemannian polyhedron, then  for $\zeta,\psi\in \dom(\Delta)$ we have the following Stokes theorem  on $Y\backslash Y^{n-2}$ (see Theorem $5.1$ in \cite{C79}),
\begin{eqnarray}\label{stokes}
\int_{Y\backslash Y^{n-2}}\Delta \zeta\cdot \psi~d\mu_h=-\int_{Y\backslash Y^{n-2}}\langle\nabla \zeta,\nabla \psi\rangle~d\mu_h.
\end{eqnarray}
Also, every closed smooth pseudomanifold $(Y,h)$ such that $h$ is equivalent to some piecewise flat metric is admissible (in the sense of Cheeger).
\end{remark}
 Now we construct the closed Riemannian polyhedron $\overline {Y}_j\subset X$ as following: Let $Y_j$ be an arbitrary Riemannian polyhedron containing $B(q,R+j)$. We consider its double  $\tilde{Y}_j$ and equip it with a Riemannian metric $\tilde{g}_j$, which is the same as Riemannian metric on $Y_j$. The Riemannian polyhedron $\overline{Y}_j=Y_j\cup \tilde{Y}_j$ with the metric $\overline{g}_j$ is an admissible closed Riemannian pseudomanifold. The metric $g_j$ on $Y_j$ is equivalent to  piecewise flat metric $g^e$ (see \cite{F01}, Chapter $4$) and so $\bar{Y}_j$ is admissible.

 We extend $\rho_j$ to $\overline{Y}_j$ such that it is zero on the copy of $Y_j$ and  $f_j$, $\eta_j$ such that they are the same functions on the copy of $Y_j$. The function $f_j$ is in  $W_{\loc}^{1,2}(\overline{Y}_j)$ (see Theorem $1.12.3.$ in \cite{KS93}).

By applying formula (\ref{stokes}) on $\overline{Y}_j$, for the functions $f_j$ and $\xi_j$, we obtain
\begin{eqnarray*}
\int_{M_j}\Delta f_j\cdot\xi_j~d\mu_{g_j}&=&-\int_{M_j}\langle\nabla f_j,\nabla\xi_j\rangle~d\mu_{g_j}
\end{eqnarray*}
where $M_j=\overline{Y}_j\backslash\overline{Y}_j^{n-2}$. Since $\xi_j\in \lip_c(M)\cap Y_j$, we can write the above Stokes formula as follows
 \begin{eqnarray*}
  \int_{M}\Delta f_j\cdot\xi_j~d\mu_{g}&=&-\int_{M}\langle\nabla f_j,\nabla\xi_j\rangle~d\mu_{g}.
 \end{eqnarray*}
\end{step}
\begin{step}\normalfont
In this step we prove that $f$ is harmonic on $M$.  From the fact that $\supp(\xi_j)\subset K_j$ we have
\begin{eqnarray*}
\int_{M}\Delta f_j\cdot\xi_j~d\mu_{g}&=&-\int_{M}\langle\nabla f_j,\nabla(\rho_j\cdot\eta_j)\rangle~d\mu_{g}\\
&=&-\int_{M}|\langle\nabla f_j,\eta_j\cdot(\nabla\rho_j)\rangle~d\mu_{g}-\int_{M}\langle\nabla f_j,\rho_j\cdot(\nabla\eta_j)\rangle~d\mu_{g}\\
 &\leq&\int_{K_j}|\nabla f_j||\nabla \rho_j|~d\mu_g+\int_{K_j}|\nabla f_j|^2~d\mu_{g}\cdot\int_{K_j}|\nabla \eta_j|^2~d\mu_{g}\\
 &\leq&\frac{1}{j}\int_{M}|\nabla f|~d\mu_g+\frac{1}{j}\int_{M}|\nabla f|^2~d\mu_{g}.
\end{eqnarray*}
So we have
\begin{eqnarray}\label{har-inq}
0\leq\int_{M}\Delta f_j\cdot\xi_j~d\mu_g\leq\frac{1}{j}(A_2+ A_1).
\end{eqnarray}
Let $h$ be any smooth function with compact support in $M\cap B(q,R)$. Then there is a $K_m$ such that the support of $h$ is in $B(q,R)\cap K_m$.  For $j$ large enough we will have $\xi_j\equiv 1$ on $K_m$ and so we have
\begin{eqnarray*}
 0\leq \int_{B(q,R)\cap K_m}\Delta f_j~d\mu_g\leq \frac{1}{j}(A_2+  A_1).
\end{eqnarray*}
 considering the formula (\ref{stokes}) as above,  for $j$ large enough we have
 \begin{eqnarray*}
0\leq\int_{M}\Delta h\cdot f_j~d\mu_g&=&\int_{M} h\cdot\Delta f_j~d\mu_g\\
&\leq& (\sup h)\cdot\frac{1}{j}(A_2+ A_1).
 \end{eqnarray*}
  Letting  $j$ go to infinity, we got  $\int_{M} \Delta h\cdot f~d\mu_g=0$. By use of Weyl's lemma $f$ is a smooth harmonic function on $M$.
\end{step}
\begin{step}\normalfont
 In this step we prove $f$ is locally Lipschitz. Since $f\in W_{\loc}^{1,2}(X)$ and by Theorem 12.2 in \cite{B11}, $f$ is harmonic on $X$. Then by Corollary 6 in \cite{K13} (see also Theorem 3.1 in \cite{J13}) $f$ is locally Lipschitz.
\begin{remark}
In \cite{KRS03}, the Lipschitz regularity of harmonic functions has been proved on metric measure spaces under the assumptions of Ahlfors regularity of the measure, Poincar\'e inequality and a heat semigroup type curvature condition. In the most recent work of \cite{K13,J13,J11} the Lipschitz regularity of the functions whose Laplacian are either in $L^p$ or in $L^{\infty}$ has been studied under more relaxed assumption on the measure. Furthermore the Cheng-Yau gradient estimate has been obtained in \cite{HKX13} for metric measure spaces under $\RCD(K,N)$ curvature dimension condition. See \cite{G13,ZhangZhu12} for the equivalent results on Alexandrov spaces.
\end{remark}
\end{step}
\begin{step}\normalfont
 Now we show $f$ is constant. Since $M$ has non-negative Ricci curvature, by the Bochner formula $|\nabla f|$ is subharmonic on $M$ and so on $X$ (see Theorem 12.2 in \cite{B11}). By Lemma \ref{yau1}, $|\nabla f|$ is constant. Since the $L^2$-norm  of $|\nabla f|$ is finite we have $|\nabla f|\equiv 0$. By Lemma \ref{bishop}, $f$ should be constant.
 \end{step}
\end{proof}

\subsection{Vanishing Results for Harmonic Maps on Complete Smooth PseudoManifolds}\label{vanishing}
In this subsection we prove Theorems \ref{rigid4} and \ref{rigid5}.
\begin{proof}[Proof of Theorem \ref{rigid4}]
 By Remark \ref{ric-man}, we know that on the Riemannian manifold $M=X\backslash S$ we have non-negative Ricci curvature. We show that for $\epsilon>0$, $\sqrt{e(u)+\epsilon}$ is weakly subharmonic on $X$. As the restriction maps $u=u{|}_M :(M,g)\rightarrow Y$ is harmonic, we  have a Bochner type formula for harmonic map on $M$ and
\begin{eqnarray*}
\Delta e(u)>|B(u)|^2
\end{eqnarray*}
where $B(u)$ is the second fundamental form of the map $u$. Also by Cauchy-Schwarz we have
\begin{eqnarray*}
|\nabla e(u)|^2\leq 2e(u)|B(u)|^2
\end{eqnarray*}
and so for $\epsilon>0$, on $X\backslash S$
\begin{eqnarray*}
\Delta \sqrt{e(u)+\epsilon}\geq 0.
\end{eqnarray*}
See e.g. the calculation in \cite{X96} Theorem $1.3.8$. Thus $\sqrt{e(u)+\epsilon}$  is subharmonic on $X\backslash S$ and by Theorem 12.2 in \cite{B11}, subharmonicity on $X$ follows since $S$ is polar and $e(f)$ is locally bounded. Therefore
\begin{eqnarray*}
\int_X\langle\nabla \sqrt{e(u)+\epsilon},\nabla \rho \rangle~d\mu_g\leq 0\quad \quad \rho\in \lip_c(X).
\end{eqnarray*}
 As in the proof of Theorem \ref{yau1},
\begin{eqnarray}
\int_{B_{R}}|\nabla \sqrt{e(u)+\epsilon}|^2~d\mu_g \leq \frac{1}{R^2}\int_{B_{2R}}e(u)+\epsilon~d\mu_g.
\end{eqnarray}
Note that $\sqrt{e(u)+\epsilon}$  satisfies all the assumptions of the Theorem  \ref{yau1}, except the finiteness of $L^2$-norm which we do not need in this step.

Set $B'_R=B_R\backslash\{ x\in B_R, e(u)(x)=0\}$. Then
\begin{eqnarray}
\int_{B'_{R}}\frac{|\nabla( {e(u)+\epsilon})|^2}{4(e(u)+\epsilon)}d\mu_g \leq \frac{1}{R^2}\int_{B_{2R}}e(u)+\epsilon ~d\mu_g.
\end{eqnarray}
 Letting $\epsilon\rightarrow 0$ gives
\begin{eqnarray}
\int_{B'_{R}}\frac{|\nabla {e(u)}|^2}{4e(u)}~d\mu_g \leq \frac{1}{R^2}\int_{B_{2R}}e(u) ~d\mu_g,
\end{eqnarray}
and  letting $R\rightarrow \infty$ and by finiteness of the energy we have
\begin{eqnarray}
\int_{B'_{R}}\frac{|\nabla {e(u)}|^2}{4e(u)}d\mu_g \leq 0
\end{eqnarray}
which implies that $e(u)$ is constant. If $e(u)$ is not zero everywhere this means that the volume of $X$ is finite. By Lemma \ref{bishop}, this is impossible and so $u$ is constant.
\end{proof}
Now we prove Theorem \ref{rigid5}. By the following lemma, the function $d(u(\cdot),q)$, where $q$ is an arbitrary point in $Y$, is subharmonic under suitable assumption on the curvature of $Y$. We refer the reader to \cite{F01} Lemma $10.2$, for the proof.
\begin{lemme}\label{distance}
Let $(X,g)$ be an admissible Riemannian polyhedron, $g$ simplexwise smooth. Let $(Y, d_Y)$ be a simply connected complete geodesic space
of non-positive curvature, and let $u\in W^{1,2}_{\loc} (X, Y)$ be  a locally energy minimizing map. Then $u$ is a locally essentially bounded map and for any $q \in Y$, the function $d(u(\cdot),q)$ of class $ W^{1,2}_{\loc} (X, Y)$ is weakly subharmonic and in particular essentially locally bounded.
\end{lemme}
We have
\begin{proof}[Proof of Theorem \ref{rigid5}]
 According the lemma  above the function $v(x)=d(u(x),u(x_0))$ for some $x_0\in X$, is weakly subharmonic. We know that $|\nabla v|^2\leq ce({u})$, where $c$ is a constant. $v$ is a continuous subharmonic function whose gradient is bounded by an $L^1$ and $L^2$ integrable function.  According to Lemma \ref{yau1}, $v$ is a constant function and so $u$ is a constant map.
\end{proof}
\begin{remark}
Using above argument we have also showed every continuous harmonic map $u:(X,g)\rightarrow Y$ belonging to $W^{1,2}_{\loc}(X, Y)$ with $\int_M\sqrt{e(u)}~d\mu_g<\infty$, where  $(X,g,\mu_g)$ is a complete,
non-compact $n$-pseudomanifold with non-negative $n$-Ricci curvature $\CD(0,n)$ and $Y$ a simply connected, complete geodesic space of non-positive curvature, is Lipschitz continuous.
\end{remark}

\section{ $2$-Parabolic Riemannian Polyhedra}\label{parabolic}
In this last section we prove Liouville-type theorems for harmonic maps defined on a Riemannian polyhedra $X$ without any completeness or Ricci curvature bound assumption. We assume instead $X$ to be $2$-parabolic. Some of these results extend known results for the case of Riemannian manifolds. As for Riemannian manifolds, we say that a domain $\Omega\subset X$ in an admissible Riemannian polyhedra $X$ is $2$-parabolic, if $\capa(D,\Omega)=0$ for every compact set $D$ in $\Omega$, otherwise $2$-hyperbolic. A reference on this subject is \cite{GT02}, where the notion is discussed for general metric measure spaces. Our main results in this section  are Theorems \ref{rigid6} and \ref{rigid7}.  We will first need the following characterization of $2$-parabolicity.
 \begin{lemme}\label{approx}
 The domain $\Omega$ is $2$-parabolic if and only if there exists a sequence of functions $\rho_j\in \lip_c(\Omega)$ such that $0\leq \rho_j\leq 1$, $\rho_j$ converges to $1$ uniformly on every compact subset of $\Omega$ and
 \begin{eqnarray*}
 \int_{\Omega}|\nabla \rho_j|^2~d\mu_g\rightarrow 0.
 \end{eqnarray*}
 \end{lemme}
\begin{proof}
First suppose $\Omega$ is $2$-parabolic. Then every compact set $D\subset\Omega$, with nonempty interior satisfies  $\capa(D,\Omega)=0$. We choose an exhaustion $D\subset D_1\subset D_2\subset\ldots \subset\Omega$ of $\Omega$ by compact subsets
 such that $\capa(D_j,\Omega)=0$ for all $j$. Hence we can find the function $\rho_j\in \lip_c(\Omega)$ (using the fact that $\lip_c(\Omega)$ is dense in $W^{1,2}_0(\Omega)$) such that $\rho_j\equiv 1$ on $D_j$ and  $\int_{\Omega}|\nabla \rho_j|^2~d\mu_g\leq 1\slash {j^2}$. We have constructed the desired sequence $\rho_j$. \\
Conversely, suppose there exists, a sequence $\rho_j\in\lip_c(\Omega)$ with the stated properties.
Then we can find a compact subset $B\subset\Omega$ and  $j_0$ such that $\rho_j\geq 1\slash 2$ for every $j\geq j_0$. It follows that $\capa(B,\Omega)=0$
 \end{proof}
The following lemma shows that the $2$-parabolicity  remains after removing the singular set of a Riemannian polyhedron.
\begin{lemme}\label{sing-para}
If $X$ is a $2$-parabolic admissible Riemannian polyhedron and $E\subset X$ is a polar set, then $\Omega:=X\backslash E$ is $2$-parabolic.
\end{lemme}
\begin{proof}
$X$ is $2$-parabolic, so by Lemma \ref{approx}, there are an exhaustion of $X$ and a sequence of function $\rho_j\in\lip_c(X)$  such that $0\leq\rho_j \leq1$ and $\rho_j\rightarrow 1$ uniformly on each compact set, and $\int_X |\nabla \rho_j|^2~d\mu_g\rightarrow 0$. Also by Lemma \ref{polar}, there exist another sequence of functions $ \varphi_j$ with support in $X\backslash E$ such that $ \varphi_j\rightarrow 1$ on each compact set of $X\backslash E$ and $\int_X| \nabla\varphi_j|^2~d\mu_g\rightarrow 0$. The functions $\rho_j \varphi_j$ on $\Omega$ provides the condition for $2$-parabolicity in Lemma \ref{approx}.
\end{proof}
The following result is an extension of Theorem $5.2$  in \cite{H90} to admissible Riemannian polyhedra.
 \begin{prop}\label{holop1}
Let $(X,g,\mu_g)$ be $2$-parabolic admissible Riemannian polyhedron. Suppose $f$ in $W^{1,2}_{\loc}(X)$ is a positive, continuous superharmonic function on $X$. Then $f$ is constant.
\end{prop}
\begin{proof}
Since $f$ is continuous, for any $\epsilon$ and at any point $x_0$ in $X$ there exist a relatively compact neighborhood $B_0$ of $x_0$ such that $f(x)>f(x_0)-\epsilon$ on $\overline{B_0}$.
$X$ is $2$-parabolic, so $\capa(B_0,X)=0$. Consider an exhaustion of $X$ by regular domains $U_i$ such that $B_0\Subset U_1\Subset U_2\Subset \ldots\Subset X$. By Corollary $11.25$ in \cite{B11}, such exhaustion exists.

There exist functions $u_i$ which are harmonic on $U_i\backslash \overline{B_0}$,  $u_i\equiv1$ on $B_0$ and $u_i\equiv0$ on $X\backslash U_i$ (See \cite{GT01} and also Lemma $11.17$ and $11.19$ in \cite{B11}). The maximum principle  (see Theorem $5.3$ in \cite{F01} or Lemma $10.2$ in \cite{B11} for the comparison principle) implies that
\begin{align*}
\left\{ \begin{array}{l}
0\leq u_i\leq 1\\
u_{i+1}\geq u_i \quad \text{on} ~U_i.\\
\end{array} \right.
\end{align*}
Define the function $h_i=(f(x_0)-\epsilon)u_i$, we have $\lim_{i\rightarrow \infty}h_i=f(x_0)-\epsilon$. On the other hand $f\geq h_i$ on the boundary of $U_i\backslash \overline{B_0}$. By the comparison principle $f\geq h_i$ in $U_i\backslash \overline{B_0}$, so $f\geq f(x_0)-\epsilon$ on $X$. Letting $\epsilon \rightarrow 0$, we obtain $f\geq f(x_0)$ on $X$. If $f$ is non-constant, there
exist $x_1\in X$ with $f(x_1)> f(x_0)$. By the same argument we obtain $f>f(x_1)$. This is a contradiction and thus $f$ is constant.
\end{proof}
We prove  the analogue of  Theorem \ref{pseudo}, for $2$-parabolic admissible Riemannian polyhedra. 
\begin{prop}\label{para-har2}
Let $(X,g,\mu_g)$ be a $2$-parabolic pseudomanifold. Let $f$ in $W^{1,2}_{\loc}(X)$ be a continuous, weakly subharmonic function, such that $\|\nabla f\|_{L^1}$ and $\|\nabla f\|_{L^2}$ are finite. Then $f$ is harmonic.
\end{prop}
\begin{proof}
 Since $X$ is $2$-parabolic, by Lemma \ref{approx} for every compact set $D\subset X$, and an arbitrary exhaustion $D\subset D_1\subset D_2\subset\ldots \subset X$ of $X$ by compact subsets, there exist a sequence of functions $\rho_j\in \lip_c(X)$ such that $\rho_j\equiv 1$ on $D_j$  and  $\int_{X}|\nabla \rho_j|^2~d\mu_g\leq 1\slash {j^2}$.
\begin{eqnarray*}
0\leq -\int_X\langle \nabla \rho_j,\nabla f\rangle~d\mu_g&\leq&\left(\int_X|\nabla \rho_j|^2~d\mu_g\right)^{\frac{1}{2}} \left(\int_X|\nabla f|^2~d\mu_g\right)^{\frac{1}{2}}\\
&\leq&\frac{1}{j}\|\nabla f\|_{L^2}^2.
\end{eqnarray*}
By Lemma \ref{radon}, there is a locally finite Borel measure $\textbf{m}_f$ such that
\begin{eqnarray*}
0<\int_D \textbf{m}_f\leq \int_X\rho_j~\textbf{m}_f\leq|\int_X \langle\nabla \rho_j,\nabla f  \rangle~d\mu_g|\leq\frac{1}{j}\|\nabla f\|_{L^2}^2.
\end{eqnarray*}
Now let $h$ be an arbitrary test function in $\lip_c(X)$ where  its support is in $D$. We have
\begin{eqnarray*}
0\leq\int_D h~ \textbf{m}_f\leq (\sup_X h)\frac{\|\nabla f\|_{L^2}^2}{j}
\end{eqnarray*}
 and so $f$ is harmonic on $X$.
\end{proof}
 Similarly we have the following result generalizing Theorem $5.9$ in \cite{H90}.
\begin{prop}\label{holop2}
Let $(X,g,\mu_g)$ be $2$-parabolic admissible Riemannian polyhedron.  Let $f$ in $W^{1,2}_{\loc}(X)$ be a harmonic function such that $\|\nabla f\|_{L^2}$ is finite. Then $f$ is constant.
\end{prop}
\begin{proof}
Set
\begin{eqnarray*}
f_i=\max(-i,\min(i,f)).
\end{eqnarray*}
Let $U_j$ be an exhaustion of $X$ by regular domains $U_j\subset U_{j+1}\Subset X$. There is a continuous function $u_{i,j}$ such that $u_{i,j}$ are harmonic on $U_j$ and $u_{i,j}=f_i$ in $X\backslash U_i$. Also $u_{i,j}$ is continuous on $X$ and $\|\nabla u_{i,j}\|_{L^2}$ is finite.   We have $-i\leq u_{i,j}\leq i$ . According to Theorem $6.2$ in  \cite{F01}, $u_{i,j}$ are H\"{o}lder continuous (after correction on a null set), and since they are uniformly bounded, by Theorem $6.3$ in \cite{F01},  they are locally uniformly H\"{o}lder equicontinuous  and by Theorem $9.37$ in \cite{B11}, there is a subsequence which converges locally uniformly to some $u_i$ as $j\rightarrow \infty$. Note that the definition of harmonicity as in \cite{B11} is consistent with our definition. The function $u_i$ is bounded and harmonic and hence is constant. Moreover $u_{i,j}-f_i\in L_{0}^{1,2}$ and so $f_i\in L_{0}^{1,2}$. Therefore
\begin{eqnarray*}
\int_X|\nabla f|^2~d\mu_g=\lim_{i\rightarrow \infty}\int_X\langle\nabla f,\nabla f_i\rangle~d\mu_g=0
\end{eqnarray*}
 and  $f$ is constant.\\
  \end{proof}
By use of Lemma \ref{distance} and the above propositions, the proofs of  Corollaries \ref{rigid6} and \ref{rigid7} are straightforward.


\begin{thebibliography}{AGS11b}
\bibitem[AGS14]{AGS12}
Luigi Ambrosio, Nicola Gigli, and Giuseppe Savar{\'e}.
\newblock Metric measure spaces with {R}iemannian {R}icci curvature bounded
  from below.
\newblock {\em Duke Math. J.}, 163(7):1405--1490, 2014.



\bibitem[AGS11a]{AGS11}
L.~{Ambrosio}, N.~{Gigli}, and G.~{Savar{\'e}}.
\newblock {Calculus and heat flow in metric measure spaces and applications to
  spaces with Ricci bounds from below}.
\newblock {\em ArXiv e-prints}, June 2011.


\bibitem[AMS13]{AMS13}
L.~{Ambrosio}, A.~{Mondino}, and G.~{Savar{\'e}}.
\newblock {On the Bakry-$\backslash$'Emery condition, the gradient estimates
  and the Local-to-Global property of $\RCD^*(K,N)$ metric measure spaces}.
\newblock {\em ArXiv e-prints}, September 2013.

\bibitem[AT04]{AT04}
Luigi Ambrosio and Paolo Tilli.
\newblock {\em Topics on analysis in metric spaces}, volume~25 of {\em Oxford
  Lecture Series in Mathematics and its Applications}.
\newblock Oxford University Press, Oxford, 2004.

\bibitem[BB11]{B11}
Anders Bj{\"o}rn and Jana Bj{\"o}rn.
\newblock {\em Nonlinear potential theory on metric spaces}, volume~17 of {\em
  EMS Tracts in Mathematics}.
\newblock European Mathematical Society (EMS), Z\"urich, 2011.

\bibitem[BBI01]{BBI01}
Dmitri Burago, Yuri Burago, and Sergei Ivanov.
\newblock {\em A course in metric geometry}, volume~33 of {\em Graduate Studies
  in Mathematics}.
\newblock American Mathematical Society, Providence, RI, 2001.

\bibitem[BS10]{BS10}
Kathrin Bacher and Karl-Theodor Sturm.
\newblock Localization and tensorization properties of the curvature-dimension
  condition for metric measure spaces.
\newblock {\em J. Funct. Anal.}, 259(1):28--56, 2010.

\bibitem[Che80]{C79}
Jeff Cheeger.
\newblock On the {H}odge theory of {R}iemannian pseudomanifolds.
\newblock In {\em Geometry of the {L}aplace operator ({P}roc. {S}ympos. {P}ure
  {M}ath., {U}niv. {H}awaii, {H}onolulu, {H}awaii, 1979)}, Proc. Sympos. Pure
  Math., XXXVI, pages 91--146. Amer. Math. Soc., Providence, R.I., 1980.

\bibitem[Che95]{C95}
Jingyi Chen.
\newblock On energy minimizing mappings between and into singular spaces.
\newblock {\em Duke Math. J.}, 79(1):77--99, 1995.

\bibitem[Che99]{C99}
J.~Cheeger.
\newblock Differentiability of {L}ipschitz functions on metric measure spaces.
\newblock {\em Geom. Funct. Anal.}, 9(3):428--517, 1999.

\bibitem[DM08]{DM08}
Georgios Daskalopoulos and Chikako Mese.
\newblock Harmonic maps from a simplicial complex and geometric rigidity.
\newblock {\em J. Differential Geom.}, 78(2):269--293, 2008.

\bibitem[DM10]{DM10}
Georgios Daskalopoulos and Chikako Mese.
\newblock Harmonic maps between singular spaces {I}.
\newblock {\em Comm. Anal. Geom.}, 18(2):257--337, 2010.

\bibitem[EF01]{F01}
J.~Eells and B.~Fuglede.
\newblock {\em Harmonic maps between {R}iemannian polyhedra}, volume 142 of
  {\em Cambridge Tracts in Mathematics}.
\newblock Cambridge University Press, Cambridge, 2001.

\bibitem[EKS13]{EKS13}
M.~{Erbar}, K.~{Kuwada}, and K.-T. {Sturm}.
\newblock {On the Equivalence of the Entropic Curvature-Dimension Condition and
  Bochner's Inequality on Metric Measure Spaces}.
\newblock {\em ArXiv e-prints}, March 2013.

\bibitem[{Gig}12]{G12}
N.~{Gigli}.
\newblock {On the differential structure of metric measure spaces and
  applications}.
\newblock {\em ArXiv e-prints}, May 2012.

\bibitem[GKO13]{G13}
Nicola Gigli, Kazumasa Kuwada, and Shin-Ichi Ohta.
\newblock Heat flow on {A}lexandrov spaces.
\newblock {\em Comm. Pure Appl. Math.}, 66(3):307--331, 2013.

\bibitem[GS92]{GS92}
Mikhail Gromov and Richard Schoen.
\newblock Harmonic maps into singular spaces and {$p$}-adic superrigidity for
  lattices in groups of rank one.
\newblock {\em Inst. Hautes \'Etudes Sci. Publ. Math.}, (76):165--246, 1992.

\bibitem[GT01]{GT01}
Vladimir Gol{\cprime}dshtein and Marc Troyanov.
\newblock Axiomatic theory of {S}obolev spaces.
\newblock {\em Expo. Math.}, 19(4):289--336, 2001.

\bibitem[GT02]{GT02}
Vladimir Gol{\cprime}dshtein and Marc Troyanov.
\newblock Capacities in metric spaces.
\newblock {\em Integral Equations Operator Theory}, 44(2):212--242, 2002.

\bibitem[GW79]{GW79}
R.~E. Greene and H.~Wu.
\newblock Smooth approximations of convex, subharmonic, and plurisubharmonic
  functions.
\newblock {\em Ann. Sci. \'Ecole Norm. Sup. (4)}, 12(1):47--84, 1979.

\bibitem[Haj96]{H96}
Piotr Haj{\l}asz.
\newblock Sobolev spaces on an arbitrary metric space.
\newblock {\em Potential Anal.}, 5(4):403--415, 1996.

\bibitem[Hil82]{H82}
Stefan Hildebrandt.
\newblock Liouville theorems for harmonic mappings, and an approach to
  {B}ernstein theorems.
\newblock In {\em Seminar on {D}ifferential {G}eometry}, volume 102 of {\em
  Ann. of Math. Stud.}, pages 107--131. Princeton Univ. Press, Princeton, N.J.,
  1982.

\bibitem[Hil85]{H85}
Stefan Hildebrandt.
\newblock Harmonic mappings of {R}iemannian manifolds.
\newblock In {\em Harmonic mappings and minimal immersions ({M}ontecatini,
  1984)}, volume 1161 of {\em Lecture Notes in Math.}, pages 1--117. Springer,
  Berlin, 1985.

\bibitem[HJW81]{H80}
S.~Hildebrandt, J.~Jost, and K.-O. Widman.
\newblock Harmonic mappings and minimal submanifolds.
\newblock {\em Invent. Math.}, 62(2):269--298, 1980/81.

\bibitem[HK98]{HK98}
Juha Heinonen and Pekka Koskela.
\newblock Quasiconformal maps in metric spaces with controlled geometry.
\newblock {\em Acta Math.}, 181(1):1--61, 1998.

\bibitem[HK00]{HK00}
Piotr Haj{\l}asz and Pekka Koskela.
\newblock Sobolev met {P}oincar\'e.
\newblock {\em Mem. Amer. Math. Soc.}, 145(688):x+101, 2000.

\bibitem[HKX13]{HKX13}
B.~{Hua}, M.~{Kell}, and C.~{Xia}.
\newblock {Harmonic functions on metric measure spaces}.
\newblock {\em ArXiv e-prints}, August 2013.

\bibitem[Hol90]{H90}
Ilkka Holopainen.
\newblock Nonlinear potential theory and quasiregular mappings on {R}iemannian
  manifolds.
\newblock {\em Ann. Acad. Sci. Fenn. Ser. A I Math. Dissertationes}, (74):45,
  1990.
  
\bibitem[Jos94]{Jost94}
J{\"u}rgen Jost.
\newblock Equilibrium maps between metric spaces.
\newblock {\em Calc. Var. Partial Differential Equations}, 2(2):173--204, 1994.

\bibitem[Jos95]{Jost95}
J{\"u}rgen Jost.
\newblock Convex functionals and generalized harmonic maps into spaces of
  nonpositive curvature.
\newblock {\em Comment. Math. Helv.}, 70(4):659--673, 1995.

\bibitem[Jos97]{Jost97}
J{\"u}rgen Jost.
\newblock Generalized {D}irichlet forms and harmonic maps.
\newblock {\em Calc. Var. Partial Differential Equations}, 5(1):1--19, 1997.

\bibitem[Jos98]{Jost98}
J{\"u}rgen Jost.
\newblock Nonlinear {D}irichlet forms.
\newblock In {\em New directions in {D}irichlet forms}, volume~8 of {\em AMS/IP
  Stud. Adv. Math.}, pages 1--47. Amer. Math. Soc., Providence, RI, 1998.

\bibitem[Jia12]{J11}
Renjin Jiang.
\newblock Lipschitz continuity of solutions of {P}oisson equations in metric
  measure spaces.
\newblock {\em Potential Anal.}, 37(3):281--301, 2012.

\bibitem[Jia14]{J13}
Renjin Jiang.
\newblock Cheeger-harmonic functions in metric measure spaces revisited.
\newblock {\em J. Funct. Anal.}, 266(3):1373--1394, 2014.

\bibitem[{Kel}13]{K13}
M.~{Kell}.
\newblock {A Note on Lipschitz Continuity of Solutions of Poisson Equations in
  Metric Measure Spaces}.
\newblock {\em ArXiv e-prints}, July 2013.

\bibitem[KRS03]{KRS03}
Pekka Koskela, Kai Rajala, and Nageswari Shanmugalingam.
\newblock Lipschitz continuity of {C}heeger-harmonic functions in metric
  measure spaces.
\newblock {\em J. Funct. Anal.}, 202(1):147--173, 2003.

\bibitem[KS93]{KS93}
Nicholas~J. Korevaar and Richard~M. Schoen.
\newblock Sobolev spaces and harmonic maps for metric space targets.
\newblock {\em Comm. Anal. Geom.}, 1(3-4):561--659, 1993.

\bibitem[KS97]{KorevaarSchoen97}
Nicholas~J. Korevaar and Richard~M. Schoen.
\newblock Global existence theorems for harmonic maps to non-locally compact
  spaces.
\newblock {\em Comm. Anal. Geom.}, 5(2):333--387, 1997.

\bibitem[KS01]{KS01}
Kazuhiro Kuwae and Takashi Shioya.
\newblock On generalized measure contraction property and energy functionals
  over {L}ipschitz maps.
\newblock {\em Potential Anal.}, 15(1-2):105--121, 2001.
\newblock ICPA98 (Hammamet).

\bibitem[KS03]{KS03}
Kazuhiro Kuwae and Takashi Shioya.
\newblock Sobolev and {D}irichlet spaces over maps between metric spaces.
\newblock {\em J. Reine Angew. Math.}, 555:39--75, 2003.

\bibitem[KS08]{S08}
Kazuhiro Kuwae and Karl-Theodor Sturm.
\newblock On a {L}iouville type theorem for harmonic maps to convex spaces via
  {M}arkov chains.
\newblock In {\em Proceedings of {RIMS} {W}orkshop on {S}tochastic {A}nalysis
  and {A}pplications}, RIMS K\^oky\^uroku Bessatsu, B6, pages 177--191. Res.
  Inst. Math. Sci. (RIMS), Kyoto, 2008.

\bibitem[LV07]{LV07}
John Lott and C{\'e}dric Villani.
\newblock Weak curvature conditions and functional inequalities.
\newblock {\em J. Funct. Anal.}, 245(1):311--333, 2007.

\bibitem[LV09]{V09}
John Lott and C{\'e}dric Villani.
\newblock Ricci curvature for metric-measure spaces via optimal transport.
\newblock {\em Ann. of Math. (2)}, 169(3):903--991, 2009.

\bibitem[Oht07]{O07}
Shin-ichi Ohta.
\newblock On the measure contraction property of metric measure spaces.
\newblock {\em Comment. Math. Helv.}, 82(4):805--828, 2007.

\bibitem[Pap05]{P05}
Athanase Papadopoulos.
\newblock {\em Metric spaces, convexity and nonpositive curvature}, volume~6 of
  {\em IRMA Lectures in Mathematics and Theoretical Physics}.
\newblock European Mathematical Society (EMS), Z\"urich, 2005.

\bibitem[Sha00]{S00}
Nageswari Shanmugalingam.
\newblock Newtonian spaces: an extension of {S}obolev spaces to metric measure
  spaces.
\newblock {\em Rev. Mat. Iberoamericana}, 16(2):243--279, 2000.

\bibitem[Stu06]{S06}
Karl-Theodor Sturm.
\newblock On the geometry of metric measure spaces. {II}.
\newblock {\em Acta Math.}, 196(1):133--177, 2006.

\bibitem[SY76]{SY76}
Richard Schoen and Shing~Tung Yau.
\newblock Harmonic maps and the topology of stable hypersurfaces and manifolds
  with non-negative {R}icci curvature.
\newblock {\em Comment. Math. Helv.}, 51(3):333--341, 1976.

\bibitem[ZZ12]{ZhangZhu12}
Hui-Chun Zhang and Xi-Ping Zhu.
\newblock Yau's gradient estimates on {A}lexandrov spaces.
\newblock {\em J. Differential Geom.}, 91(3):445--522, 2012.


%

\bibitem[Tro99]{T99}
M.~Troyanov.
\newblock Parabolicity of manifolds.
\newblock {\em Siberian Adv. Math.}, 9(4):125--150, 1999.

\bibitem[Xin96]{X96}
Yuanlong Xin.
\newblock {\em Geometry of harmonic maps}.
\newblock Progress in Nonlinear Differential Equations and their Applications,
  23. Birkh\"auser Boston Inc., Boston, MA, 1996.

\bibitem[Yau76]{Y82}
Shing~Tung Yau.
\newblock Some function-theoretic properties of complete {R}iemannian manifold
  and their applications to geometry.
\newblock {\em Indiana Univ. Math. J.}, 25(7):659--670, 1976.


\end{thebibliography}

\def\cprime{$'$} \def\cprime{$'$} \def\cprime{$'$} \def\cprime{$'$}

\end{document}